\documentclass{amsart}
\usepackage{graphicx}
\usepackage{latexsym}
\usepackage{amsmath,amsthm,amssymb}
\usepackage{amscd,braket,color}

\theoremstyle{definition}
\newtheorem{Def}{Definition}[section]
\newtheorem{theorem}[Def]{Theorem}
\newtheorem{lemma}[Def]{Lemma}

\newtheorem{cor}[Def]{Corollary}
\newtheorem{remark}[Def]{Remark}

\def\Diff{\operatorname{Diff}}
\def\exp{\operatorname{exp}}
\def\supp{\operatorname{supp}}
\def\sign{\operatorname{sign}}
\def\Sp{\operatorname{Sp}}
\def\Sign{\operatorname{Sign}}
\def\cl{\operatorname{cl}}
\def\scl{\operatorname{scl}}

\begin{document}

\title[On stable commutator length in hyperelliptic mapping class groups]
{On stable commutator length in hyperelliptic mapping class groups}

\author[D. Calegari]{Danny Calegari}
\address{University of Chicago, Chicago, Ill 60637 USA}
\email{dannyc@math.uchicago.edu}

\author[N. Monden]{Naoyuki Monden}
\address{Department of Mathematics, Graduate School of Science, Kyoto University, Kyoto 606-8502 Japan}
\email{n-monden@math.kyoto-u.ac.jp}

\author[M.Sato]{Masatoshi Sato}
\address{Department of Mathematics Education, Faculty of Education, Gifu University, Gifu 501-1193, Japan}
\email{msato@gifu-u.ac.jp}

\begin{abstract}
We give a new upper bound on the stable commutator length of Dehn twists in hyperelliptic mapping class groups,
and determine the stable commutator length of some elements.
We also calculate values and the defects of homogeneous quasimorphisms derived from $\omega$-signatures,
and show that they are linearly independent in the mapping class groups of pointed 2-spheres when the number of points is small.
\end{abstract}

\maketitle

\setcounter{secnumdepth}{2}
\setcounter{section}{0}

\section{Introduction}
The aim of this paper is to investigate stable commutator length
in hyperelliptic mapping class groups and in mapping class groups of pointed 2-spheres.
Given a group $G$ and an element $x\in[G,G]$,
the commutator length of $x$, denoted by $\cl_{G}(x)$, 
is the smallest number of commutators in $G$ whose product is $x$,
and the stable commutator length of $x$ is defined by the limit $\scl_{G}(x):=\lim_{n\rightarrow \infty} \cl_{G}(x^{n})/n$ (see Definition~\ref{def:scl} for details). 

We investigate stable commutator length in two groups $\mathcal{M}_0^m$ and $\mathcal{H}_g$.
Let $m$ be a positive integer greater than $3$.
Choose $m$ distinct points $\{q_i\}_{i=1}^m$ in a 2-sphere $S^2$.
Let $\Diff_+(S^2,\{q_i\}_{i=1}^m)$ denote the set of all orientation-preserving diffeomorphisms
in $S^2$ which preserve $\{q_i\}_{i=1}^m$ setwise with the $C^{\infty}$-topology.
We define the mapping class group of the $m$-pointed 2-sphere by $\mathcal{M}_0^m=\pi_0\Diff_+(S^2,\{q_i\}_{i=1}^m)$.
Let $\Sigma_{g}$ be a closed connected oriented surface of genus $g\geq 1$.
An involution $\iota :\Sigma_{g}\rightarrow \Sigma_{g}$ defined as in Figure~\ref{fig4} is called the hyperelliptic involution.
\begin{figure}[htbp]
 \begin{center}
\includegraphics*[width=8cm]{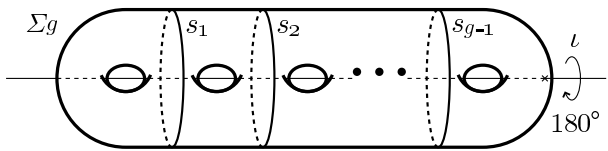}
 \end{center}
 \caption{hyperelliptic involution $\iota$ and the curves $s_{1},\ldots s_{g-1}$}
 \label{fig4}
\end{figure}
Let ${\mathcal M}_{g}$ denote the mapping class group of $\Sigma_{g}$, 
that is the group of isotopy classes of orientation-preserving diffeomorphisms of $\Sigma_{g}$,
and let ${\mathcal H}_{g}$ be the centralizer of the isotopy class of a hyperelliptic involution in ${\mathcal M}_{g}$,
which is called the hyperelliptic mapping class group of genus $g$.
Note that ${\mathcal M}_{g}={\mathcal H}_{g}$ when $g=1, \ 2$. 
Since there exists a surjective homomorphism $\mathcal{P}:\mathcal{H}_g\to \mathcal{M}_0^{2g+2}$ with finite kernel (see Lemma \ref{lemma:exact seq} and the paragraph before Remark \ref{rem:meyer function}),
these two groups have the same stable commutator length.

Let $s_{0}$ be a nonseparating curve on $\Sigma_{g}$ satisfying $\iota(s_0)=s_0$,
and let $s_{h}$ be a separating curve in Figure~\ref{fig4} for $h=1,\ldots, g-1$.
We denote by $t_{s_{j}}$ the Dehn twist about $s_{j}$ for $j=0,1,\ldots, g-1$. 
In general, it is difficult to compute stable commutator length, but those of some mapping classes are known.
In the mapping class group of a compact oriented surface with connected boundary,
Baykur, Korkmaz and the second author \cite{BKM} determined the commutator length of the Dehn twist about a boundary curve.
In the mapping class group of a closed oriented surface, interesting lower bounds on $\scl$ of Dehn twists are obtained using Gauge theory. 
Endo-Kotschick \cite{endo2001bcn}, and Korkmaz \cite{korkmaz2004scl} proved that
$1/(18g-6)\leq \scl_{{\mathcal M}_{g}}(t_{s_{j}})$ for $j=0,1, \ldots, g-1$.
For technical reasons, this result is stated in \cite{endo2001bcn} only for separating curves. 
This technical assumption is removed in \cite{korkmaz2004scl}. 
The second author \cite{monden2012ubs} also showed that $1/4(2g+1)\leq \scl_{{\mathcal H}_{g}}(t_{s_{0}})$ and 
$h(g-h)/g(2g+1)\leq \scl_{{\mathcal H}_{g}}(t_{s_{h}})$ for $h=1,\ldots,g-1$.

Stable commutator length on a group is closely related to functions on the group called homogeneous quasimorphisms through Bavard's Duality Theorem.
Homogeneous quasimorphisms are homomorphisms up to bounded error called the defect (see Definition~\ref{def:homo} for details).
By Bavard's theorem, if we obtain a homogeneous quasimorphism on the group and calculate its defect, we also obtain a lower bound on stable commutator length.
Actually, Bestvina and Fujiwara \cite[Theorem 12]{bestvina2002bcs} proved that
the spaces of homogeneous quasimorphisms on $\mathcal{M}_g$ and $\mathcal{M}_0^m$ are infinite dimensional
when $g\ge2$ and $m\ge 5$, respectively.
However, it is hard to compute explicit values of these quasimorphisms and their defects.
To compute stable commutator length,
we consider computable quasimorphisms derived from $\omega$-signature in Gambaudo-Ghys' paper \cite{gambaudo2005bs} on symmetric mapping class groups.

In Section~\ref{section:construction},
we review symmetric mapping class groups,
which are defined by Birman-Hilden as generalizations of hyperelliptic mapping class groups.
We reconsider cobounding functions of $\omega$-signatures as a series of quasimorphisms $\phi_{m,j}$ on a symmetric mapping class group $\pi_0C_g(t)$.
Since there exists a surjective homomorphism $\mathcal{P}:\pi_0C_g(t)\to \mathcal{M}_0^m$ with finite kernel, 
the homogenizations $\bar{\phi}_{m,j}$ induce homogeneous quasimorphisms on $\mathcal{M}_0^m$. 
Let $\sigma_i\in\mathcal{M}_0^m$ be a half twist which permutes the $i$-th point and the $(i+1)$-th point. 
We denote by $\tilde{\sigma}_i\in\pi_0C_g(t)$ a lift of $\sigma_i$, which will be defined in Section \ref{section:symmetric mcg}.

In Section~\ref{section:calc-quasi},
we calculate $\phi_{m,j}$ and their homogenizations $\bar{\phi}_{m,j}$.
\begin{theorem}\label{theorem:meyer function}
Let $r$ be an integer such that $2\le r\le m$.
Then, we have
\begin{enumerate}
\def\theenumi{\roman{enumi}}
\item
\[
\phi_{m,j}(\tilde{\sigma}_1\cdots\tilde{\sigma}_{r-1})=\frac{2(r-1)j(m-j)}{m(m-1)},
\]
\item
\[
\bar{\phi}_{m,j}(\sigma_1\cdots\sigma_{r-1})=-\frac{2}{r}\left\{\frac{jr(m-j)(m-r)}{m^2(m-1)}+\left(\frac{rj}{m}-\left[\frac{rj}{m}\right]-\frac{1}{2}\right)^2-\frac{1}{4}\right\},
\]
where $[x]$ denotes the greatest integer $\le x$.
\end{enumerate}
\end{theorem}
Since it requires straightforward and lengthy calculations to prove it, we leave it until the last section.
A computer calculation shows that the $([m/2]-1)\times ([m/2]-1)$ matrix whose $(i,j)$-entry is $\bar{\phi}_{m,j+1}(\tilde{\sigma}_1\cdots\tilde{\sigma}_i)$ is nonsingular when $4\le m\le 30$.
Thus, we have:
\begin{cor}
The set $\{\bar{\phi}_{m,j}\}_{j=2}^{[m/2]}$ is linearly independent when $4\le m\le 30$.
\end{cor}

In Section~\ref{section:meyer function}, we calculate the defects of the homogenizations of these quasimorphisms.
In particular, we determine the defect of $\bar{\phi}_{m,m/2}$ when $m$ is even.
Actually, $\bar{\phi}_{m,m/2}$ is the same as the homogenization of Meyer function on the hyperelliptic mapping class group $\mathcal{H}_g$.
\begin{theorem}\label{thm:defect}
Let $D(\phi_{m,j})$ and $D(\bar{\phi}_{m,j})$ be the defects of the quasimorphisms $\phi_{m,j}$ and $\bar{\phi}_{m,j}$, respectively.
\begin{enumerate}
\def\theenumi{\roman{enumi}}
\item 
For $j=1,2,\ldots,[m/2]$,
\[
D(\bar{\phi}_{m,j})\le D({\phi}_{m,j})\le m-2.
\]
\item When $m$ is even and $j=m/2$,
\[
D(\bar{\phi}_{m,m/2})=m-2.
\]
\end{enumerate}
\end{theorem}
\begin{remark}
If $\phi:G\to\mathbb{R}$ is a quasimorphism and $\bar{\phi}:G\to\mathbb{R}$ is its homogenization,
they satisfy
\[
D(\bar{\phi})\le 2D(\phi)
\]
(see \cite{calegari2009scl} Corollary 2.59).
We will claim in Lemma \ref{lem:upper bound}, when $\phi$ is antisymmetric and a class function,
they satisfy the sharper inequality
\[
D(\bar{\phi})\le D(\phi).
\]
\end{remark}

Note that when $g=2$, the hyperelliptic mapping class group $\mathcal{H}_2$ coincides with $\mathcal{M}_2$.
We may think of the lift of $\sigma_i\in\mathcal{M}_0^6$ for $i=1,2,3,4, 5$ to $\mathcal{M}_2$ as the dehn twist $t_{c_i}$ along the simple closed curve $c_i$ in Figure~\ref{fig1} (see Section \ref{section:symmetric mcg}).
Similarly, the Dehn twist $t_{s_1}\in \mathcal{M}_2$ can be considered as a lift of $(\sigma_1\sigma_2)^6\in\mathcal{M}_0^6$ by the chain relation (see Lemma \ref{lem:chain}).
Since Theorem~\ref{theorem:meyer function} (ii) implies $\bar{\phi}_{6,2}((\sigma_1\sigma_2)^6)=-8/5$
and Theorem~\ref{thm:defect} (i) implies $D(\bar{\phi}_{6,2})\le 4$,
by applying Bavard's duality theorem, we have:
\begin{cor}
\[
\frac{1}{5}\le \scl_{\mathcal{M}_2}(t_{s_1}).
\]
\end{cor}

By Theorem~\ref{thm:defect} (ii),
we can also determine the stable commutator length of the following element in ${\mathcal H}_{g}$.
\begin{cor}\label{cor}
Let $d_{2}^{+},d_{2}^{-},\ldots, d_{g-1}^{+},d_{g-1}^{-}$ be simple closed curves in Figure~\ref{figure:cd}.
Let $c$ be a nonseparating simple closed curve satisfying $\iota(c)=c$ which is disjoint from 
$d_{i}^{+},d_{i}^{-}$ and $s_{h}$ $(i=1,\ldots, g$, $h=1,\ldots,g-1)$ .
For $g\geq 2$,
\begin{eqnarray*}
\scl_{{\mathcal H}_{g}}(t_{c}^{2g+8}(t_{d_{2}^{+}}t_{d_{2}^{-}} \cdots 
t_{d_{g-1}^{+}}t_{d_{g-1}^{-}})^{2}(t_{s_{1}}\cdots t_{s_{g-1}})^{-1})=\frac{1}{2}.
\end{eqnarray*}
In particular, if $g=2$, then we have $\scl_{{\mathcal H}_{2}}(t_{c}^{12}t_{s_{1}}^{-1})=1/2$.
\end{cor}

Next, we consider upper bounds on stable commutator length.
Korkmaz also gave the upper bound $\scl_{{\mathcal M}_{g}}(t_{s_{0}})\leq 3/20$ for $g\geq 2$
(see \cite{korkmaz2004scl}).
In the case of $g=2$,
the second author showed $\scl_{{\mathcal M}_{2}}(t_{s_{0}})<\scl_{{\mathcal M}_{2}}(t_{s_{1}})$ (see \cite{monden2012ubs}).
However, these upper bounds do not depend on $g$.
On the other hand,
Kotschick proved that there is an estimate $\scl_{{\mathcal M}_{g}}(t_{s_{0}}) = O(1/g)$ 
by using the so-called ``Munchhausen trick" (see \cite{kotschick2008sls} ).

In Section~\ref{upper bound} we give the following upper bounds.
\begin{theorem}\label{thm}
Let $s_{0}$ be a nonseparating curve on $\Sigma_{g}$, and let $G_{g}$ be either ${\mathcal M}_{g}$ or ${\mathcal H}_{g}$.
For all $g\geq 1$, we have
\begin{eqnarray*}
\scl_{G_{g}}(t_{s_{0}}) \leq \frac{1}{2\{2g+3+(1/g)\}}.
\end{eqnarray*}
\end{theorem}

\section{Preliminaries}

\subsection{Stable commutator lengths and quasimorphisms}

Let $G$ denote a group, and let $[G,G]$ denote the commutator subgroup, which is the subgroup of $G$ generated 
by all commutators $[x,y]=xyx^{-1}y^{-1}$ for $x,y\in G$.
\begin{Def}\label{def:scl}
For $x\in [G,G]$, the commutator length $\cl_{G}(x)$ of $x$ is the least number of commutators in $G$ whose product is equal to $x$.
The stable commutator length of $x$, denoted $\scl(x)$, is the limit
\begin{eqnarray*}
\scl_{G}(x)=\lim_{n\rightarrow \infty}\frac{\cl_{G}(x^{n})}{n}.
\end{eqnarray*}

For each fixed $x$, the function $n\rightarrow \cl_{G}(x^{n})$ is non-negative and $\cl_{G}(x^{m+n})\leq \cl_{G}(x^{m})+\cl_{G}(x^{n})$. 
Hence, this limit exists.
If $x$ is not in $[G,G]$ but has a power $x^{m}$ which is in, define $\scl_{G}(x) = \scl_{G}(x^{m})/m$.
We also define $\scl_{G}(x)=\infty$ if no power of $x$ is contained in $[G,G]$
(we refer the reader to \cite{calegari2009scl} for the details of the theory of the stable commutator length).
\end{Def}
\begin{Def}\label{def:homo}
A quasimorphism is a function
$\phi:G\rightarrow \mathbb{R}$
for which there is a least constant $D(\phi)\geq 0$ such that
\begin{eqnarray*}
|\phi(xy)-\phi(x)-\phi(y)|\leq D(\phi)
\end{eqnarray*}
for all $x,y\in G$.
We call $D(\phi)$ the defect of $\phi$.
A quasimorphism is homogeneous if it satisfies the additional property
$\phi(x^{n})=n\phi({x})$
for all $x\in G$ and $n\in \mathbb{Z}$.
\end{Def}
We recall the following basic facts.
Let $\phi$ be a quasimorphism on $G$. For each $x \in G$, define
\begin{eqnarray*}
\bar{\phi}(a) : = \lim_{n\rightarrow \infty}\frac{\phi(x^{n})}{n}.
\end{eqnarray*}
The limit exists, and defines a homogeneous quasimorphism.
Homogeneous quasimorphisms have the following properties.
For example, see \cite[Section 5.5.2]{calegari2009scl} and \cite[Lemma 2.1 (1)]{kotschick2008sls}.
\begin{lemma}\label{quasi1}
Let $\phi$ be a homogeneous quasimorphism on $G$.
For all $x,y\in G$,

{\rm (a)} $\phi(x)=\phi(yxy^{-1})$,

{\rm (b)} $xy=yx\Rightarrow \phi(xy)=\phi(x)+\phi(y)$.
\end{lemma}
\begin{theorem}[Bavard's Duality Theorem {\rm [Ba]}]\label{Ba}
Let $Q$ be the set of homogeneous quasimorphisms on $G$ with positive defects.
For any $x\in [G,G]$, we have
\begin{eqnarray*}
\scl_{G}(x)=\displaystyle\sup_{\phi\in Q}\frac{|\phi(x)|}{2D(\phi)}.
\end{eqnarray*}
\end{theorem}

\subsection{Mapping class groups} \

For $g\geq 1$,
the abelianizations of the mapping class group ${\mathcal M}_{g}$ of the surface $\Sigma_g$ and its subgroup ${\mathcal H}_{g}$ are finite
(see \cite{powell1978ttm}).
Therefore,
all elements of ${\mathcal M}_{g}$ and ${\mathcal H}_{g}$ have powers that are products of commutators.
Dehn showed that the mapping class group $\mathcal{M}_g$ is generated by Dehn twists along non-separating simple closed curves. 
We review some relations between them.
Hereafter, we do not distinguish a simple closed curve in $\Sigma_g$ and its isotopy class.
\begin{lemma}\label{conj}
Let $c$ be a simple closed curve in $\Sigma_{g}$, and $f\in {\mathcal M}_{g}$.
Then, we have
\begin{eqnarray*}
t_{f(c)}=ft_{c}f^{-1}.
\end{eqnarray*}
\end{lemma}
From this lemma,
the values of $\scl$ and homogeneous quasimorphisms on the Dehn twists about nonseparating simple closed curves are constant.

\begin{lemma}\label{braid}
Let $c$ and $d$ be simple closed curves in $\Sigma_{g}$.
\begin{enumerate}
\def\theenumi{\alph{enumi}}
\item If $c$ is disjoint from $d$, then $t_{c}t_{d}=t_{d}t_{c}$.
\item If $c$ intersects $d$ in one point transversely, then $t_{c}t_{d}t_{c}=t_{d}t_{c}t_{d}$.
\end{enumerate}
\end{lemma}
\begin{figure}[htbp]
 \begin{center}
\includegraphics*[width=7cm]{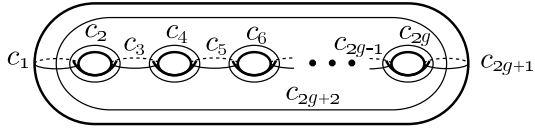}
 \end{center}
 \caption{The curves $c_{1},c_{2},\ldots,c_{2g+2}$.}
 \label{fig1}
\end{figure}
\begin{lemma}[The hyperelliptic involution]
Let $c_{1},c_{2},\ldots, c_{2g+1}$ be nonseparating curves in $\Sigma_{g}$ as in Figure~\ref{fig1}.
We call the product
\begin{eqnarray*}
\iota:= \ t_{c_{2g+1}} t_{c_{2g}}\cdots t_{c_{2}} t_{c_{1}} t_{c_{1}}t_{c_{2}}\cdots t_{c_{2g}}t_{c_{2g+1}}
\end{eqnarray*}
the hyperelliptic involution.
For $g=1$, the hyperelliptic involution $\iota$ equals to $t_{c_{1}}t_{c_{2}}t_{c_{1}}t_{c_{1}}t_{c_{2}}t_{c_{1}}$,
where $c_{1}$ (resp. $c_{2}$) is the meridian (resp. the longitude) of $\Sigma_{1}$.
\end{lemma}
\begin{lemma}[The chain relation] \label{lem:chain}
For a positive integer $n$,
let $a_1,a_2,\ldots, a_n$ be a sequence of simple closed curves in $\Sigma_g$
such that $a_i$ and  $a_j$ are disjoint if $|i-j|\geq 2$,
and that $a_i$ and $a_{i+1}$ intersect at one point.

When $n$ is odd, a regular neighborhood of $a_1\cup a_2\cup \cdots \cup  a_n$ is a subsurface of genus $(n-1)/2$ with
two boundary components, denoted by $d_1$ and $d_2$. We then have
\begin{eqnarray*}
(t_{a_n}\cdots t_{a_2}t_{a_1} )^{n+1}=t_{d_1}t_{d_2}.
\end{eqnarray*}
When $n$ is even, a regular neighborhood of $a_1\cup a_2\cup \cdots \cup  a_n$ is a subsurface of genus $n/2$ with
connected boundary, denoted by $d$. We then have
\begin{eqnarray*}
(t_{a_n}\cdots t_{a_2}t_{a_1} )^{2(n+1)}=t_d.
\end{eqnarray*}
\end{lemma}

\subsection{Meyer's signature cocycle}\label{section:meyer cocycle}
Let $X$ be a compact oriented $(4n+2)$-manifolds for nonnegative integer $n$,
and let $\Gamma$ be a local system on $X$
such that $\Gamma(x)$ is a finite dimensional real or complex vector space for $x\in X$.
If we are given a regular antisymmetric (resp. skew-hermitian) form $\Gamma\otimes \Gamma\to \mathbb{R}$ (resp. $\Gamma\otimes \Gamma\to \mathbb{C}$),
we have a symmetric (resp. hermitian) form on $H_{2n+1}(X;\Gamma)$ as in \cite[p.12]{meyer1972slk}.
For simplicity, we only explain the complex case.
It is defined by 
\begin{align*}
H_{2n+1}(X;\Gamma)\otimes H_{2n+1}(X;\Gamma)&\cong H^{2n+1}(X,\partial X;\Gamma)\otimes H^{2n+1}(X,\partial X;\Gamma)\\
&\xrightarrow{\cup} H^{4n+2}(X,\partial X; \Gamma\otimes\Gamma)\\
&\to H^{4n+2}(X,\partial X; \mathbb{C})\\
&\xrightarrow{[X,\partial X]}\ \mathbb{C},
\end{align*}
where the first row is defined by the Poincar\'e duality,
the second row is defined by the cup product of the base space,
the third row comes from the skew-hermitian form of $\Gamma$ as above, 
and the fourth row is the evaluation by the fundamental class of $X$.
Meyer showed additivity of signatures with respect to this hermitian form
(more strongly, he showed Wall's non-additivity formula for $G$-signatures of homology groups with local coefficients).

\begin{theorem}[{\cite[Satz I.3.2]{meyer1972slk}}]\label{theorem:additivity}
Let $X$ and $\Gamma$ be as above.
Assume that $X$ is obtained by gluing
two compact oriented $2n$-manifold $X_{-}$ and $X_{+}$
along some boundary components.

Then, we have 
\[
\Sign(H_n(X;\Gamma))=
\Sign(H_n(X_{-};\Gamma|_{X_-}))+\Sign(H_n(X_{+};\Gamma|_{X_+})).
\]
\end{theorem}

Consider the case when $X$ is a pair of pants, which we denote by $P$.
Let $\alpha$ and $\beta$ be loops in $P$ as in Figure~\ref{fig:pants1.eps}.
\begin{figure}[htbp]
  \begin{center}
    \includegraphics[height=3cm]{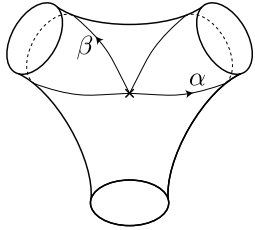}
  \end{center}
  \caption{loops in a pair of pants}
  \label{fig:pants1.eps}
\end{figure}

For $\varphi,\psi\in\mathcal{M}_g$,
there exists a $\Sigma_g$-bundle $E_{\varphi,\psi}$ on $P$ 
whose monodromies along $\alpha$ and $\beta$ are $\varphi$ and $\psi$, respectively.
This is unique up to bundle isomorphism.
In this setting, the intersection form on the local system $H_1(\Sigma_g;\mathbb{R})$ induces
the symmetric form on $H_1(P;H_1(\Sigma_g;\mathbb{R}))$.
Meyer showed that the signature of this symmetric form on $H_1(P;H_1(\Sigma_g;\mathbb{R}))$ coincides with that of $E_{\varphi,\psi}$.
Moreover, he explicitly described it in terms of the action of the mapping class group on $H_1(\Sigma_g;\mathbb{R})$ as follows.
Fix the symplectic basis $\{A_i,B_i\}_{i=1}^g$ of $H_1(\Sigma_g;\mathbb{Z})$ as in Figure~\ref{figure:symplectic},
then the action induces a homomorphism $\rho:\mathcal{M}_g\to\Sp(2g;\mathbb{Z})$.
Let $I$ denote the identity matrix of rank $g$, and $J$ denote a matrix defined by
\[
J=
\begin{pmatrix}
0&I\\
-I&0
\end{pmatrix}.
\]

For  symplectic matrices $A$ and $B$ of rank $2g$,
let $V_{A,B}$ denote the vector space defined by
\[
V_{A,B}=\{(v,w)\in \mathbb{R}^{2g}\times\mathbb{R}^{2g}\,|\,(A^{-1}-I)v+(B-I)w=0\}.
\]
Consider the symmetric bilinear form
\[
\braket{\ ,\ }_{A,B}:V_{A,B}\times V_{A,B}\to \mathbb{R}
\]
on $V_{A,B}$ defined by
\[
\braket{(v_1,w_1),(v_2,w_2)}_{A,B}:=(v_1+w_1)^T{}J(I-B)w_2.
\]
Then, the space $V_{A,B}$ coincides with $H_1(P;H_1(\Sigma_g;\mathbb{R}))$,
and the above form $\braket{\ ,\ }_{\rho(\varphi),\rho(\psi)}$ corresponds to the symmetric form on $H_1(P;H_1(\Sigma_g;\mathbb{R}))$.
\begin{figure}[htbp]
\begin{center}
\includegraphics[width=7cm]{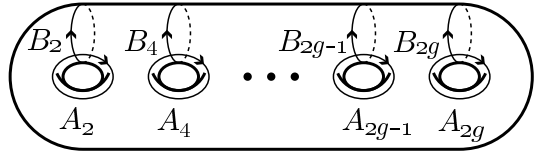}
\end{center}
\caption{A symplectic basis of $H_1(\Sigma_g;\mathbb{Z})$}
\label{figure:symplectic}
\end{figure}

Meyer's signature cocycle $\tau_g:\mathcal{M}_g\times\mathcal{M}_g\to \mathbb{Z}$ is the map defined by $(\varphi,\psi)\mapsto\braket{\ ,\ }_{\rho(\varphi),\rho(\psi)}$,
which is a bounded 2-cycle by Theorem~\ref{theorem:additivity}.
When we restrict it to the hyperelliptic mapping class group $\mathcal{H}_g$,
it represents the trivial cohomology class in $H^2(\mathcal{H}_g;\mathbb{Q})$.
Since the first homology $H_1(\mathcal{H}_g;\mathbb{Q})$ is trivial,
the cobounding function $\phi_g:\mathcal{H}_g\to\mathbb{Q}$ of $\tau_g$ is unique.
It is a quasimorphism, called the Meyer function.
In \cite{endo2000mss}, Endo computed it to investigate signatures of fibered 4-manifolds called hyperelliptic Lefschetz fibrations.
In \cite{morifuji2003mfh}, Morifuji relates it to the eta invariants of mapping tori and the Casson invariants of integral homology 3-spheres.

\section{Cobounding functions of the Meyer's signature cocycles on symmetric mapping class groups}\label{section:construction}

As in the introduction, let $m$ be a positive integer greater than $3$ and $\{q_i\}_{i=1}^m$ be $m$ distinct points in a 2-sphere $S^2$.
Choose a base point $*\in S^2-\{q_i\}_{i=1}^m$,
and denote by $\alpha_i\in\pi_1(S^2-\{q_i\}_{i=1}^m,*)$ a loop
which rounds the point $q_i$ clockwise as in Figure~\ref{fig:loop1.eps}.
\begin{figure}[htbp]
  \begin{center}
    \includegraphics[height=2cm]{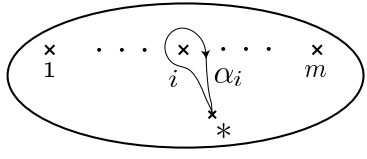}
  \end{center}
  \caption{a loop $\alpha_i$}
  \label{fig:loop1.eps}
\end{figure}

For an integer $d$ such that $d\ge2$ and $d|m$,
define a homomorphism $\pi_1(S^2-\{q_i\}_{i=1}^m)\to \mathbb{Z}/d\mathbb{Z}$
by mapping each generator $\alpha_i$ to $1\in\mathbb{Z}/d\mathbb{Z}$.
This homomorphism induces a $d$-cyclic branched covering $p_d:\Sigma_h\to S^2$ with $m$ branched points, 
where $\Sigma_h$ is a closed oriented surface of genus $h$.
Applying the Riemann-Hurwitz formula, we have $h=(d-1)(m-2)/2$.
We denote by $t:\Sigma_h\to \Sigma_h$ the deck transformation which corresponds to the generator $1\in\mathbb{Z}/d\mathbb{Z}$.

Let $\eta$ denote the $d$-th root of unity $\exp(2\pi\sqrt{-1}/d)$.
The first homology $H_1(\Sigma_h;\mathbb{C})$ decomposes into
a direct sum $\bigoplus_{j=1}^{d-1}V^{\eta^j}$, where $V^z$ is an eigenspace whose eigenvalue is $z\in\mathbb{C}$.
Note that $V^{1}$ is trivial since the quotient space $\Sigma_g/\braket{t}$ is a 2-sphere, where $\braket{t}$ denotes the cyclic group generated by the deck transformation $t$.
We also denote by $C_h(t)$ the centralizer of $t$ in the diffeomorphism group $\Diff_+\Sigma_h$.
We call the path-connected component $\pi_0C_h(t)$ the symmetric mapping class group of the covering $p$,
which is defined by Birman-Hilden (see, for example, \cite{birman1973ihr}).

In this section, we define 2-cocycles on the symmetric mapping class group $\pi_0C_h(t)$ using the $\omega$-signature in \cite{gambaudo2005bs} which derive from the nonadditivity formula.

Let us consider an oriented $\Sigma_h$-bundle $E_{\varphi,\psi}$ over $P$
whose structure group is contained in $C_h(t)$,
and monodromies along $\alpha$ and $\beta$ are $\varphi$ and $\psi$ in $\pi_0C_h(t)$, respectively.
Since coordinate transformations commute with the deck transformation $t$,
we can define a fiberwise $\mathbb{Z}/d\mathbb{Z}$ action on $E_{\varphi,\psi}$.
Since the structure group is in $C_h(t)$,
not only $H_1(\Sigma_h;\mathbb{C})$ but also each eigenspace $V^{\eta^j}$ is a local system on $P$.
We can extend the intersection form as a skew-hermitian form $H_1(\Sigma_h;\mathbb{C})\otimes H_1(\Sigma_h;\mathbb{C})\to \mathbb{C}$ defined by
\[
(x_1+x_2\sqrt{-1})\cdot(y_1+y_2\sqrt{-1})=x_1\cdot y_1+x_2\cdot y_2+(x_1\cdot y_2-x_2\cdot y_1)\sqrt{-1}.
\]

For $v\in V^{\eta^j}$ and $w\in V^{\eta^k}$ ($1\le j\le d-1$, $1\le k\le d-1$), 
\begin{align*}
(tv)\cdot w&=(\omega^jv)\cdot w=\omega^{-j}(v\cdot w),\\
(tv)\cdot w&=v\cdot (t^{-1}w)=v\cdot (\omega^{-k}w)=\omega^{-k}(v\cdot w).
\end{align*}
Since $\omega^{-j}$ is not equal to $\omega^{-k}$, 
we have $v\cdot w=0$.
Hence, $H_1(\Sigma_h;\mathbb{C})$ decomposes into an  orthogonal sum of subspaces $\{V^{\omega^j}\}_{j=1}^{d-1}$.
By restricting the intersection form on $H_1(\Sigma_h;\mathbb{C})$ to $V^{\eta^j}$,
we can define a hermitian form on $H_1(P;V^{\eta^j})$.
By Theorem~\ref{theorem:additivity}, we have a 2-cocycle on $\pi_0C_h(t)$ as follows.

\begin{lemma}
Let $j$ be an integer such that $1\le j\le m-1$.
The map $\tau_{m,d,j}:\pi_0C_h(t)\times \pi_0C_h(t)\to\mathbb{Z}$ defined by
\[
\tau_{m,d,j}(\varphi,\psi)=\Sign(H_1(P; V^{\eta^{j}}))
\]
is a $2$-cocycle,
where $V^{\eta^{j}}$ is the local system on $P$
induced from the oriented $\Sigma_h$-bundle $E_{\varphi,\psi}\to P$.
\end{lemma}

\begin{proof}
The proof is the same as for \cite[p.43 equation (0)]{meyer1972slk}.
Applying additivity of signatures to two oriented $\Sigma_h$-bundles on $P$,
we can see that $\tau_{m,d,j}$ satisfies
\[
\tau_{m,d,j}(\varphi_1,\varphi_2)+\tau_{m,d,j}(\varphi_1\varphi_2,\varphi_3)
=\tau_{m,d,j}(\varphi_1,\varphi_2\varphi_3)+\tau_{m,d,j}(\varphi_2,\varphi_3),
\]
for $\varphi_1,\varphi_2,\varphi_3\in\pi_0C_h(t)$.
\end{proof}

Since the deck transformation $t$ acts on $H^1(P,\partial P; V^{\eta^{j}})$ by multiplication of $\eta^{j}$,
we can calculate $\mathbb{Z}/d\mathbb{Z}$-signature as
\[
\Sign(H_1(P;V^{\eta^{j}}), t^k)
=\eta^{kj}\Sign(H_1(P; V^{\eta^{j}}))
=\eta^{kj}\tau_{m,d,j}(\varphi,\psi),
\]
for $0\le k\le m-1$.
Moreover, in \cite[Satz I.2.2]{meyer1972slk}, Meyer proved $\Sign(E_{\varphi,\psi},t^k)=\Sign(H_1(P;H^1(\Sigma_h;\mathbb{C})),t^k)$.
Hence, we have:

\begin{lemma}
For $0\le k\le m-1$,
\[
\Sign(E_{\varphi,\psi}, t^k)=\sum_{j=1}^{d-1}\eta^{kj}\tau_{m,d,j}(\varphi,\psi).
\]
\end{lemma}

\subsection{The symmetric mapping class groups}\label{section:symmetric mcg}
A diffeomorphism $f:\Sigma_h\to\Sigma_h$ in $C_h(t)$ induces a diffeomorphism $\bar{f}: S^2\to S^2$ which satisfies the commutative diagram
\[
\begin{CD}
\Sigma_h@>f>>\Sigma_h\\
@Vp_d VV@Vp_d VV\\
S^2 @>\bar{f}>> S^2.
\end{CD}
\]
Moreover, since $\bar{f}$ satisfies $p_d^{-1}(q)=p_d^{-1}(\bar{f}(q))$ for any $q\in S^2$, we have $\bar{f}\in\Diff_+(S^2,\{q_i\}_{i=1}^m)$. 
Therefore, we have a natural homomorphism $\mathcal{P}:\pi_0C_h(t)\to \mathcal{M}_0^m$
which maps $[f]$ to $[\bar{f}]$.
By a similar way to \cite[Theorem 1]{birman1969mcc} (see also \cite[Section 5]{birman1973ihr}),
we have:

\begin{lemma}\label{lemma:exact seq}
Let $m\ge4$. The sequence 
\[
\begin{CD}
1@>>>\mathbb{Z}/d\mathbb{Z}@>>>\pi_0C_h(t)@>\mathcal{P}>>\mathcal{M}_0^m@>>>1
\end{CD}
\]
is exact.
\end{lemma}

Let $s_i: S^2\to S^2$ be a half twist of the disk which exchanges the points $q_i$ and $q_{i+1}$ as in Figure~\ref{fig:disk1.eps}.
\begin{figure}[htbp]
  \begin{center}
    \includegraphics[height=2.6cm]{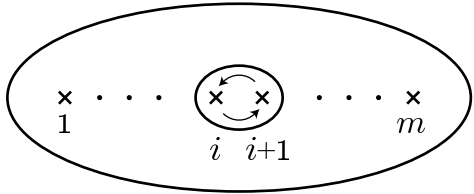}
  \end{center}
  \caption{the diffeomorphism $s_i$}
  \label{fig:disk1.eps}
\end{figure}
We denote by $\sigma_i\in\mathcal{M}_0^m$ the mapping class represented by $s_i$.
By lifting $s_i$, we have a unique diffeomorphism $\tilde{s}_i:\Sigma_h\to\Sigma_h$ which satisfies $\supp \tilde{s}_i=p_d^{-1}(\supp s_i)$.
Let us denote the path-connected component of $\tilde{s}_i$ by $\tilde{\sigma}_i$.
Note that when $d=2$, $\tilde{\sigma}_i$ is the dehn twist along a nonseparating simple closed curve.

\begin{lemma}\label{lemma:generator}
The set $\{\tilde{\sigma}_i\}_{i=1}^{m-1}\subset\pi_0C_h(t)$ generates the group $\pi_0C_h(t)$.
\end{lemma}

\begin{proof}
Since $\{\sigma_i\}_{i=1}^{m-1}$ generates the group $\mathcal{M}_0^m$,
it suffices to represent $[t]\in \pi_0C_h(t)$ as a product of $\{\sigma_i\}_{i=1}^{m-1}$.
Let $C_h^{(*)}(t)$ denote the subgroup of $C_h(t)$ defined by $C_h^{(*)}(t)=\{f\in C_h(t)\,|\, f(p_d^{-1}(*))=p_d^{-1}(*)\}$.
Only in this proof, we also call the Dehn twists the representatives of mapping classes of Dehn twists.
The diffeomorphism $s_1\cdots s_{m-2} s_{m-1}^2 s_{m-2}\cdots s_1$ in $\Diff_+(S^2,\{q_i\}_{i=1}^m)$ is isotopic to the product of Dehn twists $t_{c}^{-1}t_{c'}$ in Figure~\ref{fig:twistcurves.eps},
and it is also isotopic to the Dehn twist $t_{d}^{-1}$.
\begin{figure}[htbp]
  \begin{center}
    \includegraphics[height=2.3cm]{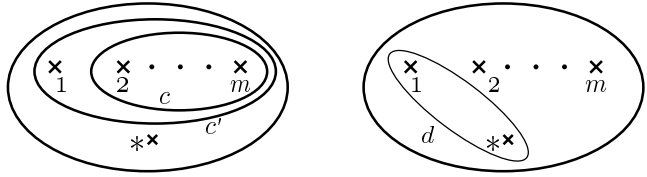}
  \end{center}
  \caption{the curves $c$, $c'$, $d$}
  \label{fig:twistcurves.eps}
\end{figure}

Therefore, the lift $\tilde{s}_1\cdots \tilde{s}_{m-2}\tilde{s}_{m-1}^2\tilde{s}_{m-2}\cdots \tilde{s}_1$ is isotopic to some lift $\tilde{f}_1:\Sigma_h\to\Sigma_h$ of $t_d^{-1}$.
Since we can choose the isotopy in $\Diff_+(S^2,\{q_i\}_{i=1}^m)$ so that it does not move $*$,
the lift $\tilde{f}_1$ fixes $p^{-1}(*)$ pointwise.
Let $D$ be the closed disk which is bounded by $d$ and contains $*$,
and $\tilde{f}_2$ denote the lift of $t_d$ which satisfies $\supp \tilde{f}_2\subset p^{-1}(D)$.
Since $f_1f_2$ is a lift of the identity map of $S^2$,
and the action of $\tilde{f}_2$ on $p^{-1}(*)$ coincides with that of $t$, we have $\tilde{f}_1\tilde{f}_2=t\in \Diff_+\Sigma_h$.
Since $t_d$ is isotopic to the identity map in $\Diff_+\Sigma_h$, we have  $[\tilde{f}_2]=1\in\pi_0C_h(t)$.
Thus, we obtain
\[
\tilde{\sigma}_1\cdots\tilde{\sigma}_{m-2}\tilde{\sigma}_{m-1}^2\tilde{\sigma}_{m-2}\cdots\tilde{\sigma}_1
=[\tilde{f}_1]=[\tilde{f}_1\tilde{f}_2]=[t]\in\pi_0C_h(t).
\]
\end{proof}

\subsection{The cobounding function of the cocycles $\tau_{m,d,j}$}\label{section:cobounding function}

Recall that, for an integer $d$ with $d|m$, we have a covering space $p_d:\Sigma_h\to S^2$.
Let $g$ be an integer $g=(m-1)(m-2)/2$.
If we consider the case when $d=m$, we also have the $m$-cyclic covering on $S^2$ whose genus of the covering surface is $g$.
thus we identify it with the surface $\Sigma_g$, and denote the covering by $p:\Sigma_g\to S^2$.

Since the quotient space $\Sigma_g/\braket{t^d}$ is also a $d$-cyclic covering of $S^2$ with $m$ branched points,
we can identify $\Sigma_h\cong\Sigma_g/\braket{t^d}$.
Since a diffeomorphism $f\in C_g(t)$ induces a diffeomorphism $\bar{f}$ on $\Sigma_g/\braket{t^{d}}$ which commutes with $t$,
we have a natural homomorphism $\mathcal{P}:\pi_0C_g(t)\to \pi_0C_h(t)$ which maps $[f]$ to $[\bar{f}]$.
Since $H^*(\pi_0C_h(t);\mathbb{Q})\cong H^*(\mathcal{M}_0^m;\mathbb{Q})$, and $H^*(\mathcal{M}_0^m;\mathbb{Q})$ is trivial (see \cite{cohen1987hmc} Corollary 2,2),
there exists a unique cobounding function of $\tau_{m,d,j}$.
Denote it by $\phi_{m,d,j}:\pi_0C_h(t)\to \mathbb{Q}$.

Actually, Gambaudo and Ghys \cite{gambaudo2005bs} already found these quasimorphisms and calculated the values of their homogenizations.
They considered the mapping class groups of pointed disks and calculated the quasimorphisms defined on them,
while we consider the mapping class groups $\mathcal{M}_0^m$ of pointed spheres.
In other words, they considered signatures of  surface bundles on a pair of pants whose fibers are surfaces with boundary instead of the closed surface $\Sigma_g$.

Let $n$ be a positive integer, and choose distinct $n$ points $\{q'_i\}_{i=1}^n$ in $D^2$. 
It is known that the mapping class group which fixes boundary pointwise and $\{q'_i\}_{i=1}^n$ setwise is isomorphic to the braid group $B_n$.
When $n\le m$, the inclusion map $(D^2, \{q_i\}_{i=1}^n) \to (S^2, \{q_i\}_{i=1}^m)$ which maps $q'_i$ to $q_i$ induces homomorphism  $\iota:B_n\to \mathcal{M}_0^m$.
Denote the half twist $\sigma'_i\in B_n$ which permutes $q'_i$ and $q'_{i+1}$.
Then, we have $\iota(\sigma'_1\sigma'_2\cdots\sigma'_{n-1})=\sigma_1\sigma_2\cdots\sigma_{n-1}$.
They calculated the value of their quasimorphisms on $\sigma'_1\sigma'_2\cdots\sigma'_{n-1}$ in Proposition 5.2.
However, their values are different from our computation of $\bar{\phi}_{m,j}(\sigma_1\sigma_2\cdots\sigma_{n-1})$ in Theorem~\ref{theorem:meyer function} since we are considering closed surfaces as fiber surfaces.

This construction is also similar to higher-order signature cocycles in Cochran-Harvey-Horn's paper \cite{cochran2012hos}.
They also considered regular coverings on a surface with boundary,
since they want to treat coverings including non-abelian and non-finite ones.
We should emphasize that the coverings we are considering are the only case
when $\pi_0 C_h(t)$ is a finite group extension of the full mapping class group $\mathcal{M}_0^m$ of a $m$-pointed sphere
with respect to a natural homomorphism $\mathcal{P}:\pi_0 C_h(t)\to \mathcal{M}_0^m$ in Section~\ref{section:symmetric mcg}.

By Birman-Hilden \cite[Theorem 1]{birman1973ihr},
the natural homomorphism $\pi_0C_h(t)\to \mathcal{M}_h$ defined by $[f]\mapsto[f]$ is injective.
In particular, if we consider the case when $m$ is even and the double covering $p_2:\Sigma_h\to S^2$, 
this homomorphism induces isomorphism between $\pi_0C_h(t)$ and $\mathcal{H}_h$.
In this case, the eigenspace $V^{-1}$ coincides with $H_1(\Sigma_h;\mathbb{C})$.
Thus, we have:
\begin{remark}\label{rem:meyer function}
When $m$ is even, $\phi_{m,2,1}:\pi_0C_h(t)\to\mathbb{Q}$ is equal to the Meyer function $\phi_h:\mathcal{H}_h\to\mathbb{Q}$ on the hyperelliptic mapping class group, under the natural isomorphism $\pi_0C_h(t)\cong \mathcal{H}_h$.
\end{remark}

\begin{lemma}\label{lemma:meyer cocycles and coverings}
For $1\le j\le d-1$ and $\varphi\in\pi_0C_g(t)$,
\[
\phi_{m,m,mj/d}(\varphi)
=\phi_{m,d,j}(\mathcal{P}(\varphi)).
\]
\end{lemma}

\begin{proof}
Since $H_1(\pi_0C_g(t);\mathbb{Q})$ is trivial,
it suffices to show that $\tau_{m,m,mj/d}(\varphi,\psi)=\tau_{m,d,j}(\mathcal{P}(\varphi),\mathcal{P}(\psi))$ for $\varphi, \psi\in\pi_0C_g(t)$.
If $f:E\to P$ is an oriented $\Sigma_g$-bundle with structure group $C_g(t)$,
the induced map $\bar{f}:E/\braket{t^{d}}\to P$ is an oriented $\Sigma_h$-bundle with structure group $C_h(t)$.
If we denote the monodromies of $f$ along $\alpha$ and $\beta$ by $\varphi$ and $\psi$, the ones of $\bar{f}$ are $\mathcal{P}(\varphi)$ and $\mathcal{P}(\psi)$.

Let $\omega$ be the $m$-th root of unity $\exp(2\pi\sqrt{-1}/m)$,
and let $q_d:\Sigma_g\to \Sigma_g/\braket{t^{d}}$ denote the projection.
To distinguish eigenspaces of $H_1(\Sigma_g;\mathbb{C})$ and $H_1(\Sigma_h;\mathbb{C})$ of the action by $t$,
we denote them by $(V_g)^z$ and $(V_h)^z$ instead of $V^z$, respectively.
The projection $q_d$ induces the isomorphism $H_1(\Sigma_g;\mathbb{C})^{\braket{t^d}}\cong H_1(\Sigma_h;\mathbb{C})$.
Moreover, we have $(V_g)^{\omega^{mj/d}}\cong (V_h)^{\eta^j}$.
Hence, it also induces a natural isomorphism between
$H_1(P;(V_g)^{\omega^{mj/d}})$ and $H_1(P;(V_h)^{\eta^j})$,
where $(V_g)^{\omega^{mj/d}}$ and $(V_h)^{\eta^j}$ are local systems coming from $f$ and $\bar{f}$.

Let $\tilde{a}$, $\tilde{b}$ be loops in $\Sigma_g-\{q_i\}_{i=1}^m$.
We may assume that $q_d(\tilde{a})\cup q_d(\tilde{b})$ has no triple point.
Then, the intersection number $[q_d(\tilde{a})] \cdot [q_d(\tilde{b})]$ in $\Sigma_h$
coincides with $[q_d^{-1}(q_d(\tilde{a}))]\cdot [\tilde{b}]$ in $\Sigma_g$.
Hence, we have
\begin{align*}
\sum_{i=0}^{m/d-1}[(t^{di})_*\tilde{a}]\cdot \sum_{j=0}^{m/d-1}[(t^{dj})_*\tilde{b}]
&=\sum_{i=0}^{m/d-1}\sum_{j=0}^{m/d-1}[(t^{{di-dj}})_*\tilde{a}]\cdot[\tilde{b}]\\
&=\frac{m}{d}[q_d^{-1}(q_d(\tilde{a}))]\cdot [\tilde{b}]\\
&=\frac{m}{d}[q_d(\tilde{a})]\cdot [q_d(\tilde{b})].
\end{align*}

Therefore, the isomorphism $H_1(\Sigma_g;\mathbb{C})^{\braket{t^d}}\cong H_1(\Sigma_h;\mathbb{C})$ induced by the quotient map $q_d:\Sigma_g\to\Sigma_h$ preserves the intersection form up to constant multiple.
Thus, it also preserves the intersection forms on $H_1(P;(V_g)^{\omega^{mj/d}})$ and $H_1(P;(V_h)^{\eta^j})$, and we obtain
\begin{align*}
\tau_{m,m,mj/d}(\varphi,\psi)
&=\Sign(H_1(P;(V_g)^{\omega^{mj/d}}))\\
&=\Sign(H_1(P;(V_h)^{\eta^j}))\\
&=\tau_{m,d,j}(\mathcal{P}(\varphi),\mathcal{P}(\psi)).
\end{align*}
\end{proof}

By Lemma~\ref{lemma:meyer cocycles and coverings},
it suffices to consider the case when $d=m$,
we simply denote $\tau_{m,m,j}=\tau_{m,j}$ and $\phi_{m,m,j}=\phi_{m,j}$.

\begin{lemma}
\[
\phi_{m,j}(\varphi)=\phi_{m,m-j}(\varphi).
\]
\end{lemma}

\begin{proof}
By taking complex conjugates,
we have an isomorphism $i:V^{\omega^j}\cong V^{\omega^{m-j}}$.
Moreover, it induces the isomorphism $i_*:H_1(P;V^{\omega^j})\cong H_1(P;V^{\omega^{m-j}})$.

Let us denote the hermitian form on $H_1(P;V^{\omega^j})$ by $\braket{\ ,\ }_j$.
By the definition of the hermitian form,
we have $\braket{x,y}_j=\overline{\braket{i_*x,i_*y}}_{m-j}$ for $x,y\in H_1(P;V^{\omega^j})$,
where $\overline{z}$ is a complex conjugate of $z\in \mathbb{C}$.
Thus, the signatures of the hermitian forms $\braket{\ ,\ }_j$ and $\braket{\ ,\ }_{m-j}$ coincide,
and the cobounding functions of $\tau_{m,j}$ and $\tau_{m,m-j}$ also coincide.
\end{proof}

\section{The homogeneous Meyer function}\label{section:meyer function}
In this section, we will prove Theorem~\ref{thm:defect}.
In Section \ref{sec:inequality},
we give an inequality between a quasimorphism and its homogenization when it is antisymmetric and a class function (Lemma \ref{lem:upper bound}),
and prove Theorem~\ref{thm:defect} (i).
In Section \ref{sec:lower bound},
we prove Theorem~\ref{thm:defect} (ii) by giving a lower bound on the defect of $\phi_{m,m/2}:\pi_0C_g(t)\to\mathbb{R}$,
which is the cobounding function of the 2-cocycle $\tau_{m,m/2}$.

\subsection{Proof of Theorem~\ref{thm:defect} (i)}\label{sec:inequality}
\begin{lemma}\label{lem:upper bound}
Let $G$ be a group, and $\phi:G\to \mathbb{R}$ a quasi-morphism satisfying
\[
\phi(xyx^{-1})=\phi(y),\ 
\phi(x^{-1})=-\phi(x).
\]
Then, we have
\[
D(\bar{\phi})\le D(\phi),
\]
where $\bar{\phi}$ is the homogenization of $\phi$.
\end{lemma}
In \cite{endo2000mss} Proposition 3.1,
Endo showed that the Meyer function $\phi_g:\mathcal{H}_g\to \mathbb{Q}$ satisfies the conditions in Lemma~\ref{lem:upper bound}. 
The quasimorphisms $\bar{\phi}_{m,j}$ also satisfy these conditions.

In \cite{turaev1987fsc},
Turaev defined another 2-cocycle on the symplectic group.
In \cite{endo2005srm} Proposition A.3,
Endo and Nagami showed that his cocycle coincides with the Meyer cocycle up to sign.
Since Turaev's cocycle is defined by the signature on a vector space of rank less than or equal to $m-2$,
A similar argument shows $D(\phi_{m,j})\le m-2$.
Thus, Theorem~\ref{thm:defect} (i) follows from Lemma~\ref {lem:upper bound}.

\begin{lemma}\label{lem:a3b3}
For any $a,b,x\in G$,
\[
\phi((ab)^3x)-\phi(xba^3b^2)=\delta\phi([b,a],(ab)^3x)+\delta\phi([a,b],babxba^2).
\]
\end{lemma}
\begin{proof}
By the definition of the coboundary operator, we have
\begin{multline}\label{eq:coboundary}
\delta\phi([b,a],(ab)^3x)+\delta\phi([a,b],babxba^2)=\phi((ab)^3x)+\phi([b,a])-\phi(ba(ab)^2x)\\
+\phi(babxba^2)+\phi([a,b])-\phi(ab^2xba^2).
\end{multline}
By the assumptions on $\phi$, we also have
\begin{align*}
\phi(ba(ab)^2x)&=\phi(babxba^2),\\
\phi([b,a])&=-\phi([a,b]),\\
\phi(ab^2xba^2)&=\phi(xba^3b^2).
\end{align*}
If we apply these equation for terms in the right-hand side of the equation (\ref{eq:coboundary}),
we obtain the desired equation.
\end{proof}

\begin{proof}[Proof of Lemma~\ref{lem:upper bound}]
Let $a$ and $b$ be elements in $G$.
Since $\delta\phi(a^{3^n},b^{3^n})=\phi(a^{3^n})+\phi(b^{3^n})-\phi(a^{3^n}b^{3^n})$, we have
\begin{align*}
\delta\bar{\phi}(a,b)&=\lim_{n\to\infty}\frac{\phi(a^{3^n})+\phi(b^{3^n})-\phi((ab)^{3^n})}{3^n}\\
&=\lim_{n\to\infty}\frac{\phi(a^{3^n}b^{3^n})-\phi((ab)^{3^n})}{3^n}.
\end{align*}
If we apply Lemma~\ref{lem:a3b3} for $x=(ab)^{3^k-3}, (ab)^{3^k-3}(ba^3b^2), \ldots, (ab)^3(ba^3b^2)^{3^{k-1}-1}, (ba^3b^2)^{3^{k-1}-1}$,
we obtain
\begin{align*}
|\phi((ab)^{3^k})-\phi((a^3b^3)^{3^{k-1}})|&=|\phi((ab)^{3^k})-\phi((ba^3b^2)^{3^{k-1}})|\\
&\le \sum_{i=0}^{3^{k-1}-1}|\phi((ab)^{3^k-3i}(ba^3b^2)^i)-\phi((ab)^{3^k-3(i+1)}(ba^3b^2)^{i+1})|\\
&\le 2D(\phi)\times 3^{k-1}.
\end{align*}
Furthermore, we have
\begin{align*}
|\phi((ab)^{3^n})-\phi(a^{3^n}b^{3^n})|&\le \sum_{k=1}^{n}|\phi((a^{3^{n-k}}b^{3^{n-k}})^{3^k})-\phi((a^{3^{n-k+1}}b^{3^{n-k+1}})^{3^{k-1}})|\\
&\le \sum_{k=1}^n(2D(\phi)\times 3^{k-1}).
\end{align*}
Hence, we obtain
\[
|\delta\bar{\phi}(a,b)|=\lim_{n\to\infty}\frac{|\phi(a^{3^n}b^{3^n})-\phi((ab)^{3^n})|}{3^n}\le D(\phi)
\]
for any $a$, $b\in G$.
\end{proof}

\subsection{Proof of Theorem~\ref{thm:defect} (ii)}\label{sec:lower bound}

We will prove Theorem~\ref{thm:defect} (ii).
Let $m$ be an even number greater than or equal to $4$.
By Remark~\ref{rem:meyer function}, we consider the Meyer function $\phi_g$ on the hyperelliptic mapping class group $\mathcal{H}_g$ instead of $\phi_{m,m/2}$.

\begin{lemma}[\cite{barge1992cem} Proposition 3.5]\label{lem:matrix rep}
For any $A\in\Sp(2g;\mathbb{Z})$,
\[
\Sign(\braket{\ ,\ }_{A^k,A})=\Sign\left(-J\sum_{i=1}^k(A^i-A^{-i})\right).
\]
\end{lemma}

Let $c_i$, $d_i^+$, and $d_i^-$ denote the Dehn twists along the simple closed curves in Figure~\ref{figure:cd}.
\begin{figure}[htbp]
\begin{center}
\includegraphics[width=6cm]{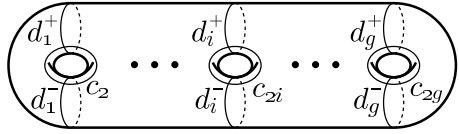}
\end{center}
\caption{curves in $\Sigma_g$}
\label{figure:cd}
\end{figure}
\begin{lemma}\label{lem:lower bound}
\[
\delta\bar{\phi_g}(c_2^2c_4^2\cdots c_{2g}^2,d_1^+d_1^-d_2^+d_2^-\cdots d_g^+d_g^-)=-2g.
\]
\end{lemma}

\begin{proof}[Proof of Lemma~\ref{lem:lower bound}]
Since the pairs ($c_i$, $c_j$), ($d_i^+d_i^-$, $d_j^+d_j^-$), and ($c_i$, $d_j^+d_j^-$) mutually commute when $i\ne j$, we have
\begin{align*}
&\delta\bar{\phi}_g(c_2^2c_4^2\cdots c_{2g}^2,d_1^+d_1^-d_2^+d_2^-\cdots d_g^+d_g^-)\\
=&\, \bar{\phi}_g(c_2^2c_4^2\cdots c_{2g}^2)+\bar{\phi}_g(d_1^+d_1^-d_2^+d_2^-\cdots d_g^+d_g^-)-\bar{\phi}_g(c_2^2d_1^+d_1^-c_4^2d_2^+d_2^-\cdots c_{2g}^2d_g^+d_g^-)\\
=&\, \sum_{i=1}^g(\bar{\phi}_g(c_{2i}^2)+\bar{\phi}_g(d_i^+d_i^-)-\bar{\phi}_g(c_{2i}^2d_i^+d_i^-)).
\end{align*}
Hence, It suffices to prove $\bar{\phi}_g(c_{2i}^2)+\bar{\phi}_g(d_i^+d_i^-)-\bar{\phi}_g(c_{2i}^2d_i^+d_i^-)=-2$ for $1\le i\le g$.
Since $\rho(d_i^+)=\rho(d_i^-)$, we have
\begin{align*}
&\bar{\phi}_g(c_{2i}^2)+\bar{\phi}_g(d_i^+d_i^-)-\bar{\phi}_g(c_{2i}^2d_i^+d_i^-)\\
=\,&-\lim_{n\to\infty}\frac{1}{n}\left\{\phi_g((c_{2i}^2d_i^+d_i^-)^n)-\phi_g((c_{2i}^2)^n)-\phi_g((d_i^+d_i^-)^n)\right\}\\
=\,&\lim_{n\to\infty}\frac{1}{n}\sum_{k=1}^{n-1}\left\{\tau_g((c_{2i}^2d_i^+d_i^-)^k,c_{2i}^2d_i^+d_i^-)-\tau_g(c_{2i}^{2i},c_{2i}^2)-\tau_g((d_i^+d_i^-)^i,d_i^+d_i^-)\right\}+\tau_g(c_{2i}^2,d_i^+d_i^-)\\
=\,&\lim_{n\to\infty}\frac{1}{n}\sum_{k=1}^{n-1}\left\{\tau_g((c_{2i}^2(d_i^+)^2)^k,c_{2i}^2(d_i^+)^2)-\tau_g(c_{2i}^{2i},c_{2i}^2)-\tau_g((d_i^+)^{2i},(d_i^+)^2)\right\}+\tau_g(c_{2i}^2,(d_i^+)^2).
\end{align*}
There exists a mapping class $\psi_i$ such that $\psi_i c_{2i}\psi_i^{-1}=c_2$ and $\psi_i d_{i}^+\psi_i^{-1}=d_i^+$ for $i=2,\ldots, g$.
Since the Meyer cocycle satisfies the property
\[
\tau_g(xyx^{-1},xzx^{-1})=\tau_g(y,z)
\]
for $x,y,z\in\mathcal{M}_g$,  we have
\begin{align*}
&\lim_{n\to\infty}\sum_{k=1}^{n-1}\frac{1}{n}\{\tau_g((c_{2i}^2(d_i^+)^2)^k,c_{2i}^2(d_i^+)^2)-\tau_g(c_{2i}^{2i},c_{2i}^2)-\tau_g((d_i^+)^{2i},(d_i^+)^2)\}+\tau_g(c_{2i}^2,(d_i^+)^2)\\
=\,&\lim_{n\to\infty}\sum_{k=1}^{n-1}\frac{1}{n}\{\tau_g((c_2^2(d_1^+)^2)^k,c_2^2(d_1^+)^2)-\tau_g(c_2^{2i},c_2^2)-\tau_g((d_1^+)^{2i},(d_1^+)^2)\}+\tau_g(c_2^2,(d_1^+)^2).
\end{align*}
Let us consider the case when $g=1$. Since
\[
\rho(c_2^2)=
\begin{pmatrix}
1&2\\
0&1
\end{pmatrix}
,\ 
\rho((d_1^+)^2)=
\begin{pmatrix}
1&0\\
-2&1
\end{pmatrix}
,\text{ and }
\rho(c_2^2(d_1^+)^2)=
\begin{pmatrix}
-3&2\\
-2&1
\end{pmatrix},
\]
we have
\begin{align*}
-J\sum_{k=1}^n(\rho(c_2^{2k})-\rho(c_2^{-2k}))
&=
\begin{pmatrix}
0&0\\
0&2n(n+1)
\end{pmatrix},
\\
-J\sum_{k=1}^n(\rho(d_1^+)^{2k}-\rho(d_1^+)^{-2k})
&=
\begin{pmatrix}
2n(n+1)&0\\
0&0
\end{pmatrix},
\\
-J\sum_{k=1}^n(\rho((c_2^2(d_1^+)^2)^k)-\rho((c_2^2(d_1^+)^2)^{-k}))
&=
\sum_{k=1}^n4k(-1)^k
\begin{pmatrix}
-1&1\\
1&-1
\end{pmatrix}\\
&=\{(-1)^n(2n+1)-1\}
\begin{pmatrix}
-1&1\\
1&-1
\end{pmatrix},
\end{align*}
By Lemma 1.1, we obtain
\begin{eqnarray}
&&\lim_{n\to\infty}\sum_{k=1}^{n-1}\frac{\tau_g(c_2^{2k},c_2^2)}{n}=\lim_{n\to\infty}\sum_{k=1}^{n-1}\frac{\tau_g((d_1^+)^{2k},(d_1^+)^2)}{n}=1,\label{eq:cocycle} \\
&&\lim_{n\to\infty}\sum_{k=1}^{n-1}\frac{\tau_g((c_2^2(d_1^+)^2)^k,c_2^2(d_1^+)^2)}{n}=0.
\end{eqnarray}
When $g\ge2$, the same calculation also shows the equation (\ref{eq:cocycle}).
It is an easy calculation to show that
\[
\tau_g(c_2^2,d_1^+d_1^-)=0.
\]
Therefore, we obtain
\[
\bar{\phi}_g(c_{2i}^2)+\bar{\phi}_g(d_i^+d_i^-)-\bar{\phi}_g(c_{2i}^2d_i^+d_i^-)=-2.
\]
\end{proof}

In the same way as the equation (\ref{eq:cocycle}), we have $\tau_g(s_0^{i},s_0)=1$.
Hence, we obtain
\begin{eqnarray*}
\bar{\phi}_g(s_0)&=&-\lim_{n\rightarrow \infty}\frac{\sum_{i=1}^{n-1}\tau_g(s_0^{i},s_0)}{n}+\phi_g(s_0)=-1+\phi_g(s_0), \ {\rm and}\\
\bar{\phi}_g(s_h)&=&\phi_g(s_0).
\end{eqnarray*}
By Lemma 3.3 and 3.5 in Endo \cite{endo2000mss}, we have 
\[
\bar{\phi}_g(t_{s_0})=-\frac{g}{2g+1},
\text{ and } 
\bar{\phi}_g(t_{s_h})=-\frac{4h(g-h)}{2g+1}.
\]

\begin{remark}\label{rem}
By Theorem~\ref{thm:defect} and~\ref{Ba}, $\bar{\phi}_{g}$ gives the lower bounds for $\scl_{{\mathcal H}_{g}}(t_{s_{h}})$ $(j=0,\ldots,g-1)$ corresponding to ones given in \cite{monden2012ubs}.
\end{remark}
\begin{remark}
By Theorem~\ref{thm}, Theorem~\ref{Ba}, and remark~\ref{rem}, 
we have $\scl_{{\mathcal M}_{1}}(t_{c})=1/12$.
Let $\rho : {\mathcal M}_{1} \cong SL(2,\mathbb{Z}) \rightarrow PSL(2,\mathbb{Z})$ be the natural quotient map.
It is easily seen that for all $x\in {\mathcal M}_{1}$, $\scl_{{\mathcal M}_{1}}(x)=\scl_{PSL(2,\mathbb{Z})}(\rho(x))$.
Louwsma determined $\scl_{PSL(2,\mathbb{Z})}(y)=1/12$ for $y=\rho(t_c)$ (see \cite{louwsma2011erq}). 
\end{remark}

\begin{proof}[Proof of Corollary~\ref{cor}]
From the above arguments, we have
\begin{eqnarray*}
D(\bar{\phi}_g) = 2g &=& \left|\sum_{i=1}^g(\bar{\phi}_g(c_{2i}^2)+\bar{\phi}_g(d_i^+d_i^-)-\bar{\phi}_g(c_{2i}^2 d_i^+d_i^-))\right|.
\end{eqnarray*}
On the other hand, for any homogeneous quasimorphism $\phi$ on ${\mathcal H}_{g}$, 
from Definition~\ref{def:homo} we have
\begin{eqnarray*}
D(\phi)&\geq& |\phi(c_2^2 \cdots c_{2g}^2)+\phi(d_1^+ d_1^- \cdots d_g^+ d_g^-)-\phi(c_2^2 \cdots c_{2g}^2 d_1^+ d_1^- \cdots d_g^+ d_g^-)|\\
&=&|\phi(c_2^2 \cdots c_{2g}^2)+\phi(d_1^+ d_1^-  \cdots d_g^+ d_g^-)-\phi((c_2^2 d_1^+ d_1^-) \cdots (c_{2g}^2 d_g^+ d_g^-))|\\
&=&\left|\sum_{i=1}^g(\phi(c_{2i}^2)+\phi(d_i^+d_i^-)-\phi(c_{2i}^2 d_i^+d_i^-))\right|.
\end{eqnarray*}
Therefore, we have
\begin{eqnarray}\label{element}
\frac{1}{2}=\sup_{\phi\in Q}\frac{|\sum_{i=1}^g(\phi(c_{2i}^2)+\phi(d_i^+d_i^-)-\phi(c_{2i}^2 d_i^+d_i^-))|}{2D(\phi)},
\end{eqnarray}
where $Q$ is the set of homogeneous quasimorphisms on ${\mathcal H}_{g}$ 
with positive defects.

By Lemma~\ref{lem:chain}, we have 
$(d_1^+ c_{2} d_1^-)^4=s_1$, 
$(d_i^+ c_{2i} d_i^-)^4=s_{i-1} s_i$ $(i=2,\ldots,g-1)$, and 
$(d_g^+ c_{2g} d_g^-)^4=s_{g-1}$.
Since $c_{2i}$ commutes with $s_j$,
$(c_{2} d_1^-d_1^+ c_{2} d_1^-d_1^+ )^2=s_1$, 
$(c_{2i} d_i^-d_i^+ c_{2i} d_i^-d_i^+ )^2=s_{i-1} s_i$, and 
$(c_{2g} d_g^-d_g^+ c_{2g} d_g^-d_g^+ )^2=s_{g-1}$.
By Lemma~\ref{braid}, $c_{2i} d_i^-d_i^+ c_{2i}$ commutes with $d_i^-d_i^+ $ 
for $i=1,\ldots, g$, as is easy to check.
It follows that 
$(c_{2} d_1^-d_1^+ c_{2})^2=s_1 (d_1^-d_1^+)^{-2}$, 
$(c_{2i} d_i^-d_i^+ c_{2i})^2=s_{i-1} s_i (d_i^-d_i^+)^{-2}$, and 
$(c_{2g} d_g^-d_g^+ c_{2g})^2=s_{g-1} (d_g^-d_g^+)^{-2}$.
These equations give
\begin{eqnarray*}
2\phi(c_{2}^{2} d_1^-d_1^+)&=&\phi(s_1)-2\phi(d_1^-d_1^+),\\
2\phi(c_{2i}^{2} d_i^-d_i^+)&=&\phi(s_{i-1})+\phi(s_{i})-2\phi(d_i^-d_i^+), \ \ \ and\\
2\phi(c_{2g}^{2} d_g^-d_g^+)&=&\phi(s_{g-1})-2\phi(d_g^-d_g^+).
\end{eqnarray*}
Now the equation (\ref{element}) becomes
\begin{eqnarray*}
\frac{1}{2}=\sup_{\phi\in Q}\frac{|\sum_{i=1}^g(\phi(c_{2i}^2)+2\phi(d_i^+d_i^-))-\sum_{j=1}^{g-1}\phi(s_j)|}{2D(\phi)}.
\end{eqnarray*}
Let $c$ be a nonseparating simple closed curve which is disjoint from 
$s_{1},$$\ldots,$$ s_{g-1},$ 
$d_{2}^{+}$, $d_{2}^{-}$,$\ldots$, $d_{g-1}^{+}$, $d_{g-1}^{-}$.
Since $d_1^+=d_1^-$ and $d_g^+=d_g^-$ and these are the Dehn twists about nonseparating curves, 
we have
\begin{eqnarray*}
\frac{1}{2}=\sup_{\phi\in Q}\frac{|(2g+8)\phi(c)+2\sum_{i=2}^{g-1}\phi(d_i^+d_i^-)-\sum_{j=1}^{g-1}\phi(s_j)|}{2D(\phi)},
\end{eqnarray*}
By using Theorem~\ref{Ba} this completes the proof.
\end{proof}

%
\section{Proof of Theorem~\ref{thm}}\label{upper bound} 
In this section, we prove Theorem~\ref{thm}.

Let $c_1,\ldots, c_{2g+2}$ be nonseparating simple closed curves on $\Sigma_g$ as in Figure~\ref{fig1} 
and let $\phi$ be a homogeneous quasimorphism on ${\mathcal H}_g$. 
For simlicity of notation, we write $t_i$ instead of $t_{c_i}$
By $\iota=\iota^{-1}$, we have $t_{2g+1}^2 t_{2g}\cdots t_2t_1^2=(t_{2g}\cdots t_2)^{-1}$. 
Since each of the two boundary components of a regular neighborhood of $c_2\cup c_3\cup \cdots \cup c_{2g}$ is $c_{2g+2}$, 
by Lemma~\ref{lem:chain} we have $(t_{2g}\cdots t_2)^{2g}=t_{2g+2}^2$. 
Note that this relation holds in ${\mathcal H}_g$. 
Therefore, by Definition~\ref{def:homo} we have 
\begin{eqnarray}\label{eq1}
\phi(t_{2g+1}^2 t_{2g}\cdots t_2t_1^2)=-\phi(t_{2g}\cdots t_2)=-\frac{1}{g}\phi(t_{2g+2}). 
\end{eqnarray}
By Lemma~\ref{quasi1} (a) and~\ref{braid} (a) we have the following equation. 
\begin{eqnarray*}
\phi(t_{2g+1}^2 t_{2g} \cdots t_3 t_2 \underline{t_1^2}) &=& \phi( \underline{t_1^2} t_{2g+1}^2 t_{2g} \cdots t_3t_2) \ \ ({\rm by \ Lem.~\ref{quasi1}})\\
&=&\phi(t_{2g+1}^2t_{2g}\cdots t_3\underline{t_1^2}\underline{\underline{t_2}}) \ \ ({\rm by \ Lem.~\ref{braid}})\\
&=&\phi(\underline{\underline{t_2}}t_{2g+1}^2t_{2g}\cdots t_4t_3t_1^2) \ \ ({\rm by \ Lem.~\ref{quasi1}})\\
&=&\phi(t_{2g+1}^2t_{2g}\cdots t_4\underline{\underline{t_2}}\underline{t_3t_1^2}) \ \ ({\rm by \ Lem.~\ref{braid}})\\
&=&\phi(t_{2g+1}^2t_{2g}\cdots t_6t_5\underline{t_3t_1^2}\underline{\underline{t_4t_2}}) \ \ ({\rm by \ Lem.~\ref{braid} \ and~\ref{braid}})\\
&=&\phi(t_{2g+1}^2t_{2g}\cdots t_6\underline{\underline{t_4t_2}}\underline{t_5t_3t_1^2}) \ \ ({\rm by \ Lem.~\ref{braid} \ and~\ref{braid}})\\
&=&\phi(t_{2g+1}^2t_{2g}\cdots t_7\underline{t_5t_3t_1^2}\underline{\underline{t_6t_4t_2}}) \ \ ({\rm by \ Lem.~\ref{braid} \ and~\ref{braid}})\\
&=&\phi(t_{2g+1}^2t_{2g}\cdots t_8\underline{\underline{t_6t_4t_2}}\underline{t_7t_5t_3t_1^2}) \ \ ({\rm by \ Lem.~\ref{braid} \ and~\ref{braid}})\\
&=&\phi(t_{2g+1}^2t_{2g}\cdots t_9\underline{t_7t_5t_3t_1^2}\underline{\underline{t_8t_6t_4t_2}}) \ \ ({\rm by \ Lem.~\ref{braid} \ and~\ref{braid}})\\
&\vdots& \ \ \ \ \ \ \ \ \ \ \ \ \ \ \ \vdots \\
&=&\phi((t_{2g+1}^2t_{2g-1}\cdots t_5t_3t_1^2)(t_{2g}t_{2g-4}\cdots t_4t_2)).
\end{eqnarray*}

From Definition~\ref{def:homo} and the equation (\ref{eq1})
\begin{eqnarray*}
D(\phi)&\geq& |\phi((t_{2g+1}^2\cdots t_3t_1^{2})(t_{2g}\cdots t_4t_2))-
\phi(t_{2g+1}^2\cdots t_3t_1^2)-\phi(t_{2g}\cdots t_4t_2)|\\
 &=&|-\frac{1}{g}\phi(t_{2g+2})-\phi(t_{2g+1}^2\cdots t_3t_1^2)-\phi(t_{2g}\cdots t_4t_2)|,
\end{eqnarray*}
where $D(\phi)$ is the defect of $\phi$.
From Lemma~\ref{quasi1}, Lemma~\ref{conj} and Lemma~\ref{braid} we have
\begin{eqnarray*}
D(\phi)\geq\left|\frac{1}{g}\phi(t_1)+(g+3)\phi(t_1)+g\phi(t_1)\right|=\left(2g+3+1/g\right)|\phi(t_1)|.
\end{eqnarray*}
By Theorem~\ref{Ba} we have $\displaystyle \scl_{{\mathcal H}_g}(t_1)\leq \frac{1}{2(2g+3+1/g)}$.
This completes the proof of Theorem~\ref{thm}.
\begin{remark}\rm
By a similar argument to the proof of Theorem~\ref{thm},
for all $m\geq 4$,
we can show $\displaystyle \scl_{{\mathcal M}_0^m}(\sigma_1)=\frac{1}{2\{m+1+2/(m-2)\}}$. 
\end{remark}

\section{Calculation of quasimorphisms}\label{section:calc-quasi}
In this section, we prove Theorem~\ref{theorem:meyer function}.
To prove it, we perform a straightforward and elementary calculation of the Hermitian form $\braket{\ ,\ }_{\tilde{\sigma}^k,\tilde{\sigma}}$ on the eigenspace $V^{\omega^j}$.

Let $p:\Sigma_g\to S^2$ be the regular branched $m$-cyclic covering on $S^2$ with $m$ branched points as in Section~\ref{section:cobounding function}.
Choose a point in $p^{-1}(*)$, and denote it by $\tilde{*}\in \Sigma_g$. 
We denote by $\tilde{\alpha}_i$ the lift of $\alpha_i$ which starts at $\tilde{*}$. 
Note that $\tilde{\alpha}_i\tilde{\alpha}_{i+1}^{-1}$ is a loop in $\Sigma_g$ while $\tilde{\alpha}_i$ is an arc.
We denote by $e_i(k)\in H_1(\Sigma_g;\mathbb{Z})$ the homology class represented by $\tilde{\alpha}_1^k\tilde{\alpha}_i\tilde{\alpha}_{i+1}^{-1}\tilde{\alpha}_1^{-k}$.
\begin{lemma}
The set of the homology classes $\{e_i(k)\}_{\begin{subarray}{c}1\le i\le m-2\\0\le k\le m-2\end{subarray}}$ is a basis of $H_1(\Sigma_g;\mathbb{Z})$.
\end{lemma}

\begin{proof}
We use the Schreier method.
Let $T$ denote a Schreier transversal $T=\{\alpha_1^k\}_{i=0}^{m-1}$,
and $S$ a generating set $S=\{\alpha_i\}_{i=1}^{m-1}$ of $\pi_1(S^2-\{q_i\}_{i=1}^m)$.
The subgroup $\pi_1(\Sigma_g-\{p^{-1}(q_i)\}_{i=1}^m)$ is generated by
\[
\{(rs(\overline{rs})^{-1}\,|\,r\in T, s\in S\}
=\{\alpha_1^k\alpha_i\alpha_1^{-k-1}\}_{\begin{subarray}{c}2\le i\le m-1\\ 0\le k\le m-2\end{subarray}}
\cup \{\alpha_1^{m-1}\alpha_i\}_{1\le i\le m-1}.
\]
By van Kampen's theorem, the group $\pi_1(\Sigma_g)$ is obtained by adding the relation $\alpha_i^m=1$ to $\pi_1(\Sigma_g-\{p^{-1}(q_i)\}_{i=1}^m)$.
Thus, the set $\{\alpha_1^k\alpha_i\alpha_{i+1}^{-1}\alpha_1^{-k}\}_{\begin{subarray}{c}1\le i\le m-2\\ 0\le k\le m-2\end{subarray}}$
generates the group $\pi_1(\Sigma_g)$. 
It implies that $\{e_i(k)\}_{\begin{subarray}{c}1\le i\le m-2\\0\le k\le m-2\end{subarray}}$ is a generating set of $H_1(\Sigma_g;\mathbb{Z})$.

By the Riemann-Hurwitz formula, $H_1(\Sigma_g;\mathbb{Z})$ is a free module of rank $2g=(m-1)(m-2)$,
and it is equal to the order of the set $\{e_i(k)\}_{\begin{subarray}{c}1\le i\le m-2\\0\le k\le m-2\end{subarray}}$.
Therefore, the set $\{e_i(k)\}_{\begin{subarray}{c}1\le i\le m-2\\0\le k\le m-2\end{subarray}}$ is a basis of the free module $H_1(\Sigma_g;\mathbb{Z})$.
\end{proof}

\subsection{The intersection form and the action of $\tilde{\sigma_i}$}
Let $j$ be an integer with $1\le j\le m-1$. 
Firstly,
we find a basis of $V^{\omega^j}\subset H_1(\Sigma_g;\mathbb{C})$,
and calculate intersection numbers.

\begin{lemma}\label{lem:intersection}
The intersection numbers of $\{e_i(k)\}_{\begin{subarray}{c}1\le i\le m-2\\ 0\le k\le m-2\end{subarray}}$ are
\begin{align*}
e_i(k)\cdot e_{i'}(k)&=
\begin{cases}
-1& \text{ if }i=i'-1,\\
1& \text{ if }i=i'+1,\\
0& \text{otherwise},
\end{cases}&
e_i(k)\cdot e_{i'}(k+1)&=
\begin{cases}
-1& \text{ if }i=i',\\
1& \text{ if }i=i'-1,\\
0& \text{otherwise},
\end{cases}\\
e_i(k)\cdot e_{i'}(k-1)&=
\begin{cases}
-1& \text{ if }i=i',\\
1& \text{ if }i=i'+1,\\
0& \text{otherwise},
\end{cases}&
e_i(k)\cdot e_{i'}(k')&=0\quad \text{ if }|k-k'|\ge2.
\end{align*}
\end{lemma}

\begin{proof}
We only prove the intersection $e_i(k)\cdot e_{i+1}(k+1)=1$
since the other cases are proved in the same way.

Let $l_i$ be the paths as in Figure~\ref{fig:li}.
\begin{figure}[htbp]
  \begin{center}
   \includegraphics[width=70mm]{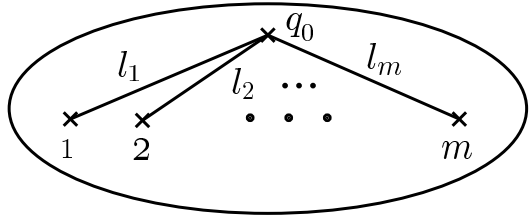}
  \end{center}
  \caption{paths $l_1,l_2,\ldots,l_m$}
  \label{fig:li}
\end{figure}
Consider the $m$-copies of the 2-sphere cut along $l_i$, and number these copies from $1$ to $m$.
For convenience, we call the first copy of the 2-sphere the $(m+1)$-th copy.
Gluing the left hand side of $l_i$ in the $k$-th copy to the right hand side of $l_i$ in the $(k+1)$-th copy for $k=1,2,\cdots, m$, 
we obtain a closed connected surface homeomorphic to $\Sigma_g$,
and it is naturally a covering space on $S^2$.
As in Figure~\ref{fig:e_ij1} and Figure~\ref{fig:e_ij2}, the loops representing $e_i(k)$ and $e_{i+1}(k+1)$,
intersect once positively in the $(k+1)$-th copy.
\begin{figure}[htbp]
\begin{tabular}{cc}
 \begin{minipage}{0.51\hsize}
  \begin{center}
   \includegraphics[width=0.85\hsize]{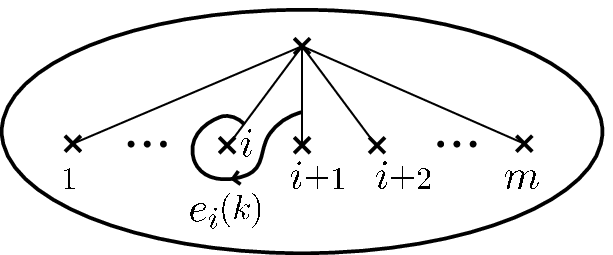}
  \end{center}
  \caption{the $k$-th copy}
  \label{fig:e_ij1}
 \end{minipage}
 &
 \begin{minipage}{0.51\hsize}
  \begin{center}
   \includegraphics[width=0.85\hsize]{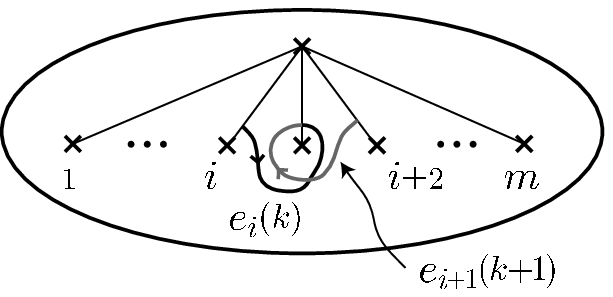}
  \end{center}
  \caption{the $(k+1)$-th copy}
  \label{fig:e_ij2}
 \end{minipage}
 \end{tabular}
\end{figure}

Hence, we have $e_i(k)\cdot e_{i+1}(k+1)=1$.
\end{proof}

For $1\le i\le m-2$,
we denote $w_i=\sum_{k=0}^{m-1}\omega^{-jk}e_i(k)$. 
Since $te_i(k)=e_i(k+1)$ for $1\le k\le m-2$ and $e_i(m-1)=-\sum_{k=0}^{m-2}e_i(k)$,
we have $w_i\in V^{\omega^j}$,
and the set $\{w_i\}_{i=1}^{m-2}$ is a basis of $V^{\omega^j}$.

\begin{lemma}\label{lem:w-intersection}
The intersecton numbers of $\{w_i\}_{1\le i\le m-2}$ are
\[
w_i\cdot w_{i'}=
\begin{cases}
d(1-\omega^j), &\text{ if }i=i'+1,\\
d(-\omega^{-j}+\omega^{j}), &\text{ if }i=i',\\
d(\omega^{-j}-1), &\text{ if }i=i'-1,\\
0, &\text{ otherwise}.
\end{cases}
\]
\end{lemma}

\begin{proof}
By Lemma~\ref{lem:intersection}, we have
\begin{align*}
w_i\cdot w_i&=\sum_{k=0}^{d-1}\sum_{l=0}^{d-1}\omega^{j(k-l)}e_i(k)\cdot e_i(l)=d(-\omega^{-j}+\omega^{j}),\\
w_i\cdot w_{i+1}&=\sum_{k=0}^{d-1}\sum_{l=0}^{d-1}\omega^{j(k-l)}e_i(k)\cdot e_{i+1}(l)=d(\omega^{-j}-1),
\end{align*}
and $w_i\cdot w_k=0$ when $|i-k|\ge 2$.
\end{proof}

Let $\tilde{\sigma}=\tilde{\sigma}_1\cdots\tilde{\sigma}_{r-1}$.
Secondly, we will find eigenvectors in $V^{\omega^j}$ with respect to the action by $\tilde{\sigma}$.

\begin{lemma}\label{lem:act on homology}
Let $i$ be an integer such that $1\le i\le m-1$.
Then, we have
\[
(\tilde{\sigma}_i)_*e_l(k)=
\begin{cases}
e_l(k)+e_{l+1}(k),& \text{if }2\le i\le m-1\text{, and }l=i-1,\\
-e_l(k-1),& \text{if }l=i,\\
e_{l-1}(k-1)+e_l(k),& \text{if }l=i+1,\\
e_l(k),& \text{if }l\ne i-1,i,i+1.
\end{cases}
\]
\end{lemma}

\begin{proof}
Recall that $e_i(k)$ is the homology class represented by the loop $\tilde{\alpha}_1^k\tilde{\alpha}_i\tilde{\alpha}_{i+1}^{-1}\tilde{\alpha}_1^{-k}$.
In the fundamental group $\pi_1(S^2-\{q_i\}_{i=1}^m)$, we have
\begin{align*}
(\sigma_i)_*(\alpha_{i-1}\alpha_i^{-1})&=\alpha_{i-1}\alpha_{i+1}^{-1}=(\alpha_{i-1}\alpha_i^{-1})(\alpha_i\alpha_{i+1}^{-1}),\\
(\sigma_i)_*(\alpha_i\alpha_{i+1}^{-1})&=\alpha_{i+1}(\alpha_{i+1}^{-1}\alpha_i\alpha_{i+1})^{-1}=\alpha_{i+1}^{-1}(\alpha_{i}\alpha_{i+1}^{-1})^{-1}\alpha_{i+1},\\
(\sigma_i)_*(\alpha_{i+1}\alpha_{i+2}^{-1})&=(\alpha_{i+1}^{-1}\alpha_i\alpha_{i+1})\alpha_{i+2}^{-1}=\alpha_{i+1}^{-1}(\alpha_i\alpha_{i+1}^{-1})\alpha_{i+1}(\alpha_{i+1}\alpha_{i+2}^{-1}).
\end{align*}
By lifting these loops to the covering space $\Sigma_g$, 
we obtain what we want.
\end{proof}

By Lemma~\ref{lem:act on homology}, the matrix representations of the actions of $\{\tilde{\sigma}_i\}_{i=1}^{m-1}$ on $V^{\omega^j}$ with respect to the basis $\{w_i\}_{1\le i\le m-2}$ are calculated as
\begin{eqnarray*}
(\tilde{\sigma}_1)_*=
\begin{pmatrix}
-\omega^{-j}&\omega^{-j}&O\\
0&1&O\\
O&O&I_{m-4}
\end{pmatrix},&
(\tilde{\sigma}_i)_*=
\begin{pmatrix}
I_{i-1}&O&O\\
O&L&O\\
O&O&I_{m-i-4}
\end{pmatrix},\\
(\tilde{\sigma}_{m-2})_*=
\begin{pmatrix}
I_{m-4}&O&O\\
O&1&O\\
O&1&-\omega^{-j}
\end{pmatrix},&
(\tilde{\sigma}_{m-1})_*=
\begin{pmatrix}
I_{m-3}&v\\
O&-1+\sum_{k=1}^{m-2}\omega^{-jk}
\end{pmatrix},
\end{eqnarray*}
where
\[
L=
\begin{pmatrix}
1&0&0\\
1&-\omega^{-j}&\omega^{-j}\\
0&0&1
\end{pmatrix},
v=(1,1+\omega^{-j},\ldots,\sum_{k=0}^{m-3}\omega^{-jk})^T.
\]

Let $r$ be an integer with $2\le r\le m$, and put
\[
e'_r(k)=[\tilde{a}_1^{k}\tilde{a}_r(\tilde{a}_1\tilde{a}_2\cdots \tilde{a}_r)^{-1}\tilde{a}_1^{-1}(\tilde{a}_1\tilde{a}_2\cdots \tilde{a}_r)\tilde{a}_1^{-k}].
\]

By Lemma~\ref{lem:act on homology}, we have
\begin{align*}
\tilde{\sigma}_*e_i(k)&=e_{i+1}(k),\text{ when }1\le i\le r-2,\\
\tilde{\sigma}_*e_r(k)&=-e_r'(k)+e_r(k),\\
\tilde{\sigma}_*e_{r-1}(k)&=e'_r(k),\\
\tilde{\sigma}_*e'_r(k)&=e_1(k-r+1).
\end{align*}
If we put $w'_r=\sum_{k=0}^{m-1}\omega^{-jk}e'_r(k)$,
$w'_r$ is contained in $V^{\omega^j}$.
For $i=1,2,\ldots, r-2$, we have
\begin{align*}
\tilde{\sigma}_*w_i&=\tilde{\sigma}_*\sum_{k=0}^{m-1}\omega^{-jk}e_i(k)
=\sum_{k=0}^{m-1}\omega^{-jk}e_{i+1}(k)
=w_{i+1},\\
\tilde{\sigma}_*w_{r-1}
&=\sum_{k=0}^{m-1}\omega^{-jk}(e_{r-1}(k)+e'_r(k))
=w_{r-1}+w'_r,\\
\tilde{\sigma}_*w'_r
&=\sum_{k=0}^{m-1}\omega^{-jk}e_1(k-r+1)
=\sum_{k=0}^{m-1}\omega^{-j(k+r-1)}e_1(k)
=\omega^{-(r-1)j}w_1.
\end{align*}

Let $\zeta=\exp(2\pi\sqrt{-1}/r)$ and
$v_i=\sum_{k=1}^{r-1}\omega^{(k-1)j}\zeta^{-(k-1)i}w_k+\omega^{(r-1)j}\zeta^{-(r-1)i}w'_r$.
Then, we have
\begin{align*}
\tilde{\sigma}_*v_i
&=\sum_{k=1}^{r-1}\omega^{(k-1)j}\zeta^{-(k-1)i}(\tilde{\sigma})_*w_k+\omega^{(r-1)j}\zeta^{-(r-1)i}(\tilde{\sigma})_*w'_r\\
&=\sum_{k=1}^{r-2}\omega^{(k-1)j}\zeta^{-(k-1)i}w_{k+1}+\omega^{(r-2)j}\zeta^{-(r-2)i}w'_r+\omega^{(r-1)j}\zeta^{-(r-1)i}\omega^{-pj}w_1\\
&=\omega^{-j}\zeta^i\left(\sum_{k=1}^{r-1}\omega^{(k-1)j}\zeta^{-(k-1)i}w_k+\omega^{(r-1)j}\zeta^{-(r-1)i}w'_r\right)\\
&=(\omega^{-j}\zeta^i)v_i.
\end{align*}
Hence, $v_i$ is an eigenvector with eigenvalue $\omega^{-j}\zeta^{i}$ with respect to the action by $\tilde{\sigma}$.
Note that the subspace generated by $\{w_i\}_{i=1}^{r-1}$ coincides with one generated by $\{v_i\}_{i=1}^{r-1}$.
Since $\tilde{\sigma}$ acts trivially on $\{w_i\}_{i=r+1}^{m-1}$, they are also eigenvectors with eigenvalue $0$.
Moreover, the set $\{v_i\}_{i=1}^{r-1}\cup \{w_{i}\}_{i=r+1}^{m-2}$ is linearly independent.

\begin{lemma}\label{lem:intersection of v}
Let $i,i'$ be integers such that $1\le i\le r-1$ and $1\le i'\le r-1$. Then, we have
\[
v_i\cdot v_{i'}=
\begin{cases}
\displaystyle 8rd\sqrt{-1}\sin \frac{\pi i}{r} \sin \frac{\pi j}{m}\sin \pi\left(\frac{i}{r}-\frac{j}{m}\right),&\text{ if }i=i',\\
0,&\text{ otherwise}.
\end{cases}
\]
\end{lemma}

\begin{proof}
Since the action of the mapping class group $\pi_0C_g(t)$ preserves the intersection form,
\begin{align*}
v_i\cdot v_i&=\sum_{k=0}^{r-1}\sum_{l=0}^{r-1}\omega^{(l-k)j}\zeta^{-(l-k)i}(\tilde{\sigma}^k_*w_1\cdot \tilde{\sigma}^l_*w_1)\\
&=\sum_{k=0}^{r-1}\sum_{l=0}^{r-1}\omega^{(l-k)j}\zeta^{-(l-k)i}(w_2\cdot \tilde{\sigma}^{l-k+1}_*w_1).
\end{align*}
Thus, Lemma~\ref{lem:w-intersection} implies
\begin{align*}
v_i\cdot v_i&=\omega^{(r-1)j}\zeta^{-(r-1)i}(w_2\cdot\tilde{\sigma}^{r}_*w_1)+\omega^{j}\zeta^{-i}(r-1)(w_2\cdot\tilde{\sigma}^2_*w_1)+r(w_2\cdot\tilde{\sigma}_*w_1)\\
&\quad+\omega^{-j}\zeta^{i}(r-1)(w_2\cdot w_1)+\omega^{-(r-1)j}\zeta^{(r-1)i}(w_2\cdot\tilde{\sigma}^{-r+2}_*w_1)\\
&=r\{(\omega^{-j}\zeta^{i})w_2\cdot w_1+(\omega^{j}\zeta^{-i})w_2\cdot w_3+w_2\cdot w_2\}\\
&=8rd\sqrt{-1}\sin \frac{\pi i}{r} \sin \frac{\pi j}{m}\sin \pi\left(\frac{i}{r}-\frac{j}{m}\right).
\end{align*}
\end{proof}

\subsection{Calculation of $\omega$-signatures and the cobounding functions $\phi_{m,j}$}
Lastly, we will calculate the Hermitian form $\braket{\ ,\ }_{\tilde{\sigma}^k,\tilde{\sigma}}$ and the $\omega$-signature.
We have already found the set of eigenvectors $\{v_i\}_{i=1}^{r-1}\cup \{w_{i}\}_{i=r+1}^{m-2}$ with respect to the action by $\tilde{\sigma}$ which is linearly independent.
Since $\dim V^{\omega^j}=m-2$, we need to find another eigenvector.

\begin{lemma}
\[
\sum_{k=1}^{rm}\tau(\tilde{\sigma}^k,\tilde{\sigma})=rm-2|mi-rj|.
\]
\end{lemma}

\begin{proof}
We first consider the case when $rj/m$ is not an integer.
Put
\[
\beta=\sum_{i=1}^{r}w_i-\frac{1}{r}\sum_{k=1}^r\frac{1}{1-\omega^j\zeta^{-k}}v_k.
\]
The subspace generated by $\{v_i\}_{i=1}^{r-1}$
and that generated by $\{w_i\}_{i=1}^{r-1}$ coincides.
Thus, the set $\{v_i\}_{i=1}^{r-1}$, $\beta$, $\{w_i\}_{i=r+1}^{m-2}$ forms a basis of $V^{\omega^j}$ when $1\le r\le m-2$,
and the set $\{v_i\}_{i=1}^{m-2}$ forms a basis of $V^{\omega^j}$ when $r=m-1$.
We have
\begin{align*}
\tilde{\sigma}_*\beta&=\sum_{i=2}^{r}w_i-\frac{1}{r}\sum_{k=1}^r\frac{\omega^j\zeta^{-k}}{1-\omega^j\zeta^{-k}}v_k\\
&=\sum_{i=2}^{r}w_i+\frac{1}{r}\sum_{k=1}^rv_k-\frac{1}{r}\sum_{k=1}^r\frac{1}{1-\omega^j\zeta^{-k}}v_k\\
&=\sum_{i=1}^{r}w_i-\frac{1}{r}\sum_{k=1}^r\frac{1}{1-\omega^j\zeta^{-k}}v_k\\
&=\beta.
\end{align*}
Note that $\beta$ and $\{w_i\}_{i=r+1}^{m-2}$ are in the annihilator of the Hermitian form $\braket{\ ,\ }_{\tilde{\sigma}^k,\tilde{\sigma}}$ since they have eigenvalue $1$  with respect to the action by $\tilde{\sigma}$.

By Lemma~\ref{lem:matrix rep}, we have
\begin{align*}
\quad\tau(\tilde{\sigma}^k,\tilde{\sigma})
&=\sum_{i=1}^r\sign\braket{v_i,v_i}_{\tilde{\sigma}^k,\tilde{\sigma}}\\
&=-\sum_{i=1}^r\sign\left((v_i\cdot v_i)\sum_{l=1}^{k}((\omega^{-j}\zeta^{i})^l-(\omega^j\zeta^{-i})^l)\right)\\
&=-\sum_{i=1}^{r}\sign\left((v_i\cdot v_i)(1-\omega^{j}\zeta^{-i})(1-\omega^{-j}\zeta^{i})\sum_{l=1}^{k}((\omega^{-j}\zeta^{i})^l-(\omega^j\zeta^{-i})^l)\right).
\end{align*}

By the equations
\begin{align*}
&\quad(1-\omega^{j}\zeta^{-i})(1-\omega^{-j}\zeta^{i})\sum_{l=1}^{k}((\omega^{-j}\zeta^{i})^l-(\omega^j\zeta^{-i})^l)\\
&=8\sqrt{-1}\sin\left(-\frac{\pi (k+1)j}{m}+\frac{\pi (k+1)i}{r}\right)\sin\left(-\frac{\pi k j}{m}+\frac{\pi k i}{r}\right)\sin\left(-\frac{\pi j}{m}+\frac{\pi i}{r}\right)
\end{align*}
and Lemma~\ref{lem:intersection of v}, we have
\[
\quad\tau((\tilde{\sigma})^k,\tilde{\sigma})\\
=\sum_{i=1}^{r-1}\sign\left(\sin k\pi\left(\frac{i}{r}-\frac{j}{m}\right)\sin (k+1)\pi\left(\frac{i}{r}-\frac{j}{m}\right)\right).
\]

Since $rj/m$ is not an integer, $i/r-j/m$ is not zero.
Thus, we obtain
\begin{align*}
\sum_{k=1}^{rm}\tau((\tilde{\sigma})^k,\tilde{\sigma})
&=\sum_{i=1}^{r-1}\sum_{k=1}^{rm}\sign\left(\sin k\pi\left(\frac{i}{r}-\frac{j}{m}\right)\sin (k+1)\pi\left(\frac{i}{r}-\frac{j}{m}\right)\right)\\
&=\sum_{i=1}^{r-1}(rm-2|mi-rj|).
\end{align*}

Next, consider the case when $rj/m$ is an integer and $1\le r\le m-1$.
Denote this integer $rj/m$ by $i_0$. 
Then, the eigenvalue of $v_{i_0}$ is $1$,
and $v_{i_0}$ and $\{w_i\}_{i=r+1}^{m-2}$ are in the annihilator of $\braket{\ ,\ }_{\tilde{\sigma}^k,\tilde{\sigma}}$.
If we put
\[
\beta'=\sum_{i=1}^rw_i-\frac{1}{r}\sum_{\begin{subarray}{c}1\le k\le r\\k\ne i_0\end{subarray}}\frac{1}{1-\omega^j\zeta^{-k}}v_k,
\]
the set of the homology classes $\{v_i\}_{i=1}^{r-1}$, $\beta'$, $\{w_i\}_{i=r+1}^{m-2}$ forms a basis of $V^{\omega^j}$.
We have
\begin{align*}
\tilde{\sigma}\beta'&=\sum_{i=2}^{r}w_i-\frac{1}{r}\sum_{\begin{subarray}{c}1\le k\le r\\k\ne i_0\end{subarray}}\frac{\omega^j\zeta^{-k}}{1-\omega^j\zeta^{-k}}v_k\\
&=\sum_{i=2}^{r}w_i+\frac{1}{r}\sum_{\begin{subarray}{c}1\le k\le r\\k\ne i_0\end{subarray}}v_k-\frac{1}{r}\sum_{\begin{subarray}{c}1\le k\le r\\k\ne i_0\end{subarray}}\frac{1}{1-\omega^j\zeta^{-k}}v_k\\
&=\sum_{i=1}^{r}w_i-\frac{1}{r}v_{i_0}-\frac{1}{r}\sum_{\begin{subarray}{c}1\le k\le r\\k\ne i_0\end{subarray}}\frac{1}{1-\omega^j\zeta^{-k}}v_k\\
&=\beta'-\frac{1}{r}v_{i_0}.
\end{align*}
By Lemma~\ref{lem:matrix rep},
\[
\braket{\beta',\beta'}_{\tilde{\sigma}^k,\tilde{\sigma}}
=\beta'\cdot \frac{1}{r}\sum_{i=1}^k2iv_{i_0}
=\frac{k(k+1)}{r}\sum_{i=1}^r w_i\cdot v_{i_0}.
\]
Since the eigenvalues of $\{v_i\}_{i=1}^{r-1}$ are different from $1$,
the intersection $v_i\cdot v_{i_0}=0$ for $1\le i\le r-1$.
Since the subspace generated by $\{w_i\}_{i=1}^{r-1}$ and that generated by $\{v_i\}_{i=1}^{r-1}$ coincides, we also have $w_i\cdot v_{i_0}=0$.
Thus, we have
\begin{align*}
\frac{r}{k(k+1)}\braket{\beta',\beta'}_{\tilde{\sigma}^k,\tilde{\sigma}}
&=w_r\cdot v_{i_0}\\
&=w_r\cdot (w_{r-1}+w_r')\\
&=w_r\cdot \left(w_{r-1}-\sum_{k=0}^{r-1}\omega^{(k-r)j}w_k\right)\\
&=(1-\omega^{-j})w_r\cdot w_{r-1}\\
&=(1-\omega^{-j})(1-\omega^{j})>0.
\end{align*}

Moreover, since $v_i\cdot v_{i_0}=0$, Lemma~\ref{lem:matrix rep} implies $\braket{v_i, \beta'}_{\tilde{\sigma}^k,\tilde{\sigma}}=0$ for $1\le i\le r-1$.
Therefore, we have
\begin{align*}
\sum_{k=1}^{rm}\tau(\tilde{\sigma}^k,\tilde{\sigma})
&=\sum_{k=1}^{rm}\left(\sum_{i=1}^k\sign(\braket{v_i,v_i}_{\tilde{\sigma}^k,\tilde{\sigma}})
+\sign(\braket{\beta',\beta'}_{\tilde{\sigma}^k,\tilde{\sigma}})\right)\\
&=\sum_{k=1}^{rm}\left(\sum_{\begin{subarray}{c}1\le i\le r-1\\i\ne i_0\end{subarray}}\sign\left(\sin k\pi\left(\frac{i}{r}-\frac{j}{m}\right)\sin (k+1)\pi\left(\frac{i}{r}-\frac{j}{m}\right)\right)+1\right)\\
&=\sum_{\begin{subarray}{c}1\le i\le r-1\\i\ne i_0\end{subarray}}(rm-2|mi-rj|)+rm\\
&=\sum_{i=1}^{r-1}(rm-2|mi-rj|).
\end{align*}

In the case when $r=m$, the set $\{v_i\}_{i=1}^{r-2}$ forms a basis of $V^{\omega^j}$.
By a similar calculation, we can also prove what we want.
\end{proof}

\begin{lemma}\label{lemma:homogenization}
For $r=2,3,\ldots,m$, 
\[
\phi_{m,j}(\tilde{\sigma})-\bar{\phi}_{m,j}(\tilde{\sigma})
=\frac{2}{r}\left\{\left(\frac{rj}{m}-\left[\frac{rj}{m}\right]-\frac{1}{2}\right)^2-\frac{r^2j(m-j)}{m^2}-\frac{1}{4}\right\}.
\]
\end{lemma}

\begin{proof}

\[
\tau(\tilde{\sigma}^k,\tilde{\sigma})
=\sum_{i=1}^{r-1}\sign\left(\sin k\pi\left(\frac{i}{r}-\frac{j}{m}\right)\sin (k+1)\pi\left(\frac{i}{r}-\frac{j}{m}\right)\right),
\]
Since we have $\tau(\tilde{\sigma}^{k+rm},\tilde{\sigma})=\tau(\tilde{\sigma}^k,\tilde{\sigma})$,
\begin{align*}
\quad\phi_{m,j}(\tilde{\sigma})-\bar{\phi}_{m,j}(\tilde{\sigma})
&=\frac{1}{rm}\sum_{k=1}^{rm}\tau(\tilde{\sigma}^k,\tilde{\sigma})\\
&=\frac{1}{rm}\sum_{i=1}^{r-1}(rm-2|mi-rj|)\\
&=r-1-\frac{2}{rm}\left(\sum_{i=1}^{[\frac{rj}{m}]}(rj-mi)+\sum_{[\frac{rj}{m}]+1}^{r-1}(mi-rj)\right)\\
&=\frac{2}{r}\left\{\left(\frac{rj}{m}-\left[\frac{rj}{m}\right]-\frac{1}{2}\right)^2+\frac{r^2j(m-j)}{m^2}-\frac{1}{4}\right\}.
\end{align*}
\end{proof}

\begin{proof}[Proof of Theorem~\ref{theorem:meyer function}]
Applying Lemma~\ref{lemma:homogenization} to the case when $r=m$,
we have
\[
\phi_{m,j}(\tilde{\sigma}_1\cdots\tilde{\sigma}_{m-1})-\bar{\phi}_{m,j}(\tilde{\sigma}_1\cdots\tilde{\sigma}_{m-1})
=\frac{2j(m-j)}{m}.
\]
Since
\[
\bar{\phi}_{m,j}(\tilde{\sigma}_1\cdots\tilde{\sigma}_{m-1})=\frac{1}{m}\bar{\phi}_{m,j}((\tilde{\sigma}_1\cdots\tilde{\sigma}_{m-1})^m)=0,
\]
we have 
\[
\phi_{m,j}(\tilde{\sigma}_1\cdots\tilde{\sigma}_{m-1})=\frac{2j(m-j)}{m}.
\]

Put 
$\varphi=\tilde{\sigma}_1\tilde{\sigma}_3\cdots\tilde{\sigma}_{m-1}, \psi=\tilde{\sigma}_2\tilde{\sigma}_4\cdots\tilde{\sigma}_{m-2}$
when $m$ is even, and
$\varphi=\tilde{\sigma}_1\tilde{\sigma}_3\cdots\tilde{\sigma}_{m-2}, \psi=\tilde{\sigma}_2\tilde{\sigma}_4\cdots\tilde{\sigma}_{m-1}$,
when $m$ is odd. 
As we saw in Section~\ref{upper bound},
$\tilde{\sigma}_1\cdots\tilde{\sigma}_{m-1}$ is conjugate to $\varphi\psi$.
By direct computation, if $(\varphi^{-1}_*-I_{2g})x+(\psi_*-I_{2g})y=0$ for $x,y\in V^{\omega^j}$, 
we have $(\varphi^{-1}_*-I_{2g})x=(\psi_*-I_{2g})y=0$.
Hence, we have $\tau_g(\varphi,\psi)=0$.

In the same way, for $i=1,2,\ldots,[(m-1)/2]$, we have
\begin{align*}
\tau_g(\tilde{\sigma}_1\tilde{\sigma}_3\cdots\tilde{\sigma}_{2i+1},\tilde{\sigma}_2\tilde{\sigma}_4\cdots\tilde{\sigma}_{2i})
&=\tau_g(\tilde{\sigma}_1\tilde{\sigma}_3\cdots\tilde{\sigma}_{2i+1},\tilde{\sigma}_2\tilde{\sigma}_4\cdots\tilde{\sigma}_{2i+2})
=0,\\
\tau_g(\tilde{\sigma}_1\tilde{\sigma}_3\cdots\tilde{\sigma}_{2i-1},\tilde{\sigma}_{2i+1})
&=\tau_g(\tilde{\sigma}_2\tilde{\sigma}_4\cdots\tilde{\sigma}_{2i},\tilde{\sigma}_{2i+2})
=0.
\end{align*}
Thus,
\[
\phi_{m,j}(\tilde{\sigma})
=(r-1)\phi_{m,j}(\tilde{\sigma}_1)
=\frac{r-1}{m-1}\phi_{m,j}(\tilde{\sigma}_1\cdots\tilde{\sigma}_{m-1})
=\frac{2(r-1)j(m-j)}{m(m-1)}.
\]

Hence, we obtain
\begin{align*}
\bar{\phi}_{m,j}(\sigma_1\cdots\sigma_{r-1})
&=\bar{\phi}_{m,j}(\tilde{\sigma})\\
&=\phi_{m,j}(\tilde{\sigma})-(\phi_{m,j}(\tilde{\sigma})-\bar{\phi}_{m,j}(\tilde{\sigma}))\\
&=-\frac{2}{r}\left\{\frac{jr(m-j)(m-r)}{m^2(m-1)}+\left(\frac{rj}{m}-\left[\frac{rj}{m}\right]-\frac{1}{2}\right)^2-\frac{1}{4}\right\}.
\end{align*}
\end{proof}

By the values of $\bar{\phi}_{m,1}$, we see:
\begin{remark}
Let $r$ be an integer such that $2\le r\le m$.
\[
\bar{\phi}_{m,1}(\tilde{\sigma}_1\cdots\tilde{\sigma}_{r-1})=0.
\]
\end{remark}
However, we do not know whether the quasimorphism $\bar{\phi}_{m,1}$ is trivial or not.

\end{document}